\documentclass[11pt,reqno]{amsart}
\usepackage{amsmath}
\usepackage{amssymb}
\usepackage{amscd}
\usepackage{latexsym}
\usepackage{graphicx}
\usepackage{color}

\newcommand{\wh}{\widehat}
\newcommand{\ispa}[1]{\langle \,#1 \,\rangle }

\newcommand{\sgn}{\operatorname{sgn}\nolimits}
\newcommand{\ol}{\overline}
\newcommand{\mf}{\mathfrak}
\newcommand{\mb}{\mathbb}
\newcommand{\scal}{\mathcal{S}}
\newcommand{\kcal}{\mathcal{K}}
\newcommand{\lcal}{\mathcal{L}}
\newcommand{\re}{{\rm Re}\,}
\newcommand{\im}{{\rm Im}\,}
\newcommand{\vphi}{\varphi}
\newcommand{\spec}{{\rm Spec\hspace{0.1mm}}}
\newcommand{\dsp}{\displaystyle}
\newcommand{\down}{{\rm {\scriptstyle down}}}
\newcommand{\up}{{\rm {\scriptstyle up}}}

\newtheorem{thm}{{\sc Theorem}}[section]
\newtheorem{cor}[thm]{{\sc Corollary}}
\newtheorem{lem}[thm]{{\sc Lemma}}
\newtheorem{prop}[thm]{{\sc Proposition}}
\newtheorem{defin}[thm]{{\sc Definition}}
\newenvironment{rem}{\medskip\noindent{\it Remark$:$\/} }{\medskip}

\begin{document}

\title[Up and Down Grover walks]
{Up and Down Grover walks on simplicial complexes}
\author{Xin Luo}
\address{Xin Luo, College of Mathematics and Econometrics, Hunan University,Changsha 410082,China}
\email{xinluo@hnu.edu.cn}
\author{Tatsuya Tate}
\address{Tatsuya Tate, Mathematical Institute, Graduate School of Sciences, Tohoku University,
Aoba, Sendai 980-8578, Japan. }
\email{tate@m.tohoku.ac.jp}
\thanks{The second author is partially supported by JSPS Grant-in-Aid for Scientific Research (No. 25400068, No. 15H02055).}
\date{\today}

\renewcommand{\thefootnote}{\fnsymbol{footnote}}
\renewcommand{\theequation}{\thesection.\arabic{equation}}
\renewcommand{\labelenumi}{{\rm (\arabic{enumi})}}
\renewcommand{\labelenumii}{{\rm (\alph{enumii})}}
\numberwithin{equation}{section}

\begin{abstract}
A notion of up and down Grover walks on simplicial complexes are proposed and their properties are investigated.
These are abstract Szegedy walks, which is a special kind of unitary operators on a Hilbert space.
The operators introduced in the present paper are usual Grover walks on graphs
defined by using combinatorial structures of simplicial complexes.
But the shift operators are modified so that it can contain information of
orientations of each simplex in the simplicial complex.
It is well-known that the spectral structures of this kind of unitary operators
are almost determined by its discriminant operators.
It has strong relationship with combinatorial Laplacian on simplicial complexes
and geometry, even topology, of simplicial complexes.
In particular, theorems on a relation between spectrum of up and down discriminants
and orientability, on a relation between symmetry of spectrum of discriminants
and combinatorial structure of simplicial complex are given.
Some examples, both of finite and infinite simplicial complexes, are also given.
Finally, some aspects of finding probability and stationary measures are discussed.

%\noindent{Keywords:} simplicial complex; Grover walks; combinatorial Laplacian

%\noindent{Mathematical Subject Classification.} 05C50,39A12
\end{abstract}

\maketitle

\section{Introduction}\label{INTRO}

{\it Grover walks}, originally introduced in \cite{W} and named after a famous work of Grover \cite{G} on
a quantum search algorithm, is one of unitary time evolution operators, often called discrete-time
quantum walks, defined over graphs. These are introduced in computer sciences and
developed in areas of mathematics, such as probability theory, spectral theory
and geometric analysis. It was Szegedy \cite{Sz} who had realized that their spectral structure of
{\it Szegedy walks}, which generalizes Grover walks,
are almost determined by a self-adjoint operator, called {\it discriminant} operator.
Szegedy's idea also works well for infinite graphs as is developed in \cite{HKSS}, \cite{SS}.
More concretely, an abstract Szegedy walk, is a unitary operator of the form
\begin{equation}\label{G1}
U=S C,
\end{equation}
where $S$ and $C$ are unitary operators on a separable Hilbert space
satisfying $S^{2}=C^{2}=I$.

In \cite{MOS}, certain class of unitary transitions on simplicial complexes are introduced.
Suppose that $\kcal=(V,\scal)$ is a simplicial complex with certain conditions where
$V$ is a set of vertices and $\scal$ is a set of simplices.
Let $\wh{K}_{q}$ be the set of sequences of vertices of length $q+1$ which form simplices in $\scal$.
The symmetric group $\mf{S}_{q+1}$ of order $(q+1)!$ acts on $\wh{K}_{q}$ naturally.
The operators introduced in \cite{MOS} act on the Hilbert space $\ell^{2}(\wh{K}_{q})$.
They have the form $\eqref{G1}$ but the operator $S$, which is often called a `shift operator'
of an abstract Szegedy walk, is given by the action of certain permutation $\pi \in \mf{S}_{q+1}$,
and hence in general it does not satisfy $S^{2}=I$.
It seems that the operators introduced in \cite{MOS} would have rather advantage
because one can choose permutations $\pi$ for various purposes.
However, to find their geometric aspects, it does not seem so transparent, because
it is not quite clear which permutation should be chosen to relate the operators with geometry.
Simplicial complex is a geometric, topological and combinatorial object.
Hence it would be rather natural to expect that operators so-defined have geometric information.
Like Laplacians acting on differential forms, there is a notion of {\it combinatorial Laplacians}
which is defined by replacing the exterior differentials in the definition of Laplacians acting on differential forms
by the coboundary operator in simplicial cohomology theory.
A general framework for this combinatorial Laplacian was
introduced and investigated in an interesting article \cite{HJa}.
They have a rich geometric aspects, such as Hodge decomposition.

The purpose in the present paper is to introduce and investigate
other Grover walks on simplicial complexes.
The definition is rather simple. The operators we mainly consider are Grover walks
on graphs, which we call {\it up and down graphs} (and it is essentially the same as
the dual graph used in \cite{HJa}), defined by using combinatorial structures of simplicial complexes.
They are basically Grover walks on graphs but the shift operator is a bit different.
Namely it is modified from the usual shift operators on graphs, which will be necessary
to take the orientation of simplices into account.
Indeed this modification makes Grover walks on up and down graphs,
which we call {\it up and down Grover walks},
certainly have geometric aspects.
We also consider the alternating sum of the operators introduced in \cite{MOS}.
It has also certain relation with the combinatorial Laplacian. However, there is a significant difference
between this and our up and down Grover walks. This difference is caused by a lack of
the `down parts' of the operators introduced in \cite{MOS}. In contrast, our operators
are `two-folds', there are two operators having close relation with up and down Laplacians,
and hence they would have rich geometric information.
Indeed, one of our main theorem (Theorem $\ref{ori1}$) says that, for certain simplicial complexes,
the down-Grover walk in top dimension has eigenvalue $-1$
if and only if the simplicial complex has a coherent orientation.
This shows that `down parts' also have nice geometric information.

The organization of the present paper is as follows.
After preparing some notion and terminology of simplicial complexes and function spaces in Section $\ref{NT}$,
the definitions of various operators investigated in the paper will be given in Section $\ref{DEFGR}$.
In Section $\ref{SIMP}$ some of fundamental properties of up and down Grover walks
are given. One of main parts is Section $\ref{ORI}$ where one can find nice relation
between the spectrum of our operators and geometry. In this section the spectrum of our
operators for infinite cylinder is computed. In Section $\ref{SSS}$ relationship between
spectral symmetry and combinatorial structures is investigated, and some examples for finite simplicial complexes
are given. Finally, in Section $\ref{FP}$, finding probabilities defined by our operators are investigated and,
in particular, certain stationary measures are given.

\section{Notation and terminology}\label{NT}

In this section, we prepare some notation used in this paper.
Throughout the present paper, $\kcal=(V,\scal)$ is an abstract simplicial complex, or simply a simplicial complex, with
the set of vertices $V$ and the set of simplices $\scal$. We recall that the set $\scal$ is a subset in $2^{V}$
closed under inclusion, $\scal$ contains sets of the form $\{v\}$ with $v \in V$ and
each elements in $\scal$ is a finite subset of $2^{V}$.
It is assumed that the empty set is always contained in $\scal$, and $V$ is a countable set.
For a given simplex $F \in \scal$,
if the number of elements in $F$ is $q+1$, then we say that the dimension of $F$ is $q$ and
in this case we write $\dim F=q$.
The set of all simplex of the dimension $q$ is denoted by $\scal_{q}$.
We call subsets of a simplex $F \in \scal_{q}$ faces of $F$.

\subsection{Terminology on simplicial complexes}\label{TSC}

We mean by an {\it ordered simplex} in $\kcal$ the sequence $s=(a_{0}a_{1}\cdots a_{q})$ of vertices $a_{j}$ in $V$
with $\{a_{0},a_{1},\ldots,a_{q}\} \in \scal_{q}$. We say that the dimension of an ordered simplex $s=(a_{0}a_{1}\cdots a_{q})$ is $q+1$.
The set of all ordered simplex of dimension $q$ is denoted by $\wh{K}_{q}$.

Let $\pi$ be a permutation on the set $\{0,1,\ldots,q\}$ consisting of $q+1$ elements
and let $s=(a_{0} \cdots a_{q}) \in \wh{K}_{q}$.
Then we define an element $s^{\pi}$ in $\wh{K}_{q}$ by $s^{\pi}=(a_{\pi(0)} \cdots a_{\pi(q)})$.
If $\pi_{1}$ and $\pi_{2}$ are two permutation on $\{0,1,\ldots,q\}$, we have $(s^{\pi_{1}})^{\pi_{2}}=s^{\pi_{1} \pi_{2}}$.
Therefore this determines an action of the symmetric group $\mf{S}_{q+1}$ of order $(q+1)!$ on
the set of all ordered simlices $\wh{K}_{q}$ of dimension $q$.
Then we define $K_{q}$ by $K_{q}=\wh{K}_{q}/\mf{A}_{q+1}$, where $\mf{A}_{q+1}$ is the alternating group of order $(q+1)!/2$.
We note that $\mf{A}_{q+1}$ is defined as the group consisting of all the permutations in $\mf{S}_{q+1}$ with signature $1$.
We call elements in $K_{q}$ {\it oriented simplices} of dimension $q$.
An equivalence class of ordered simplex $s=(a_{0} \cdots a_{q})$ is denoted by $\ispa{s} \in K_{q}$.
Since $\mf{S}_{q+1}/\mf{A}_{q+1} \cong \mb{Z}_{2}$, we have an action of $\mb{Z}_{2}$ on $K_{q}$.
We denote this action by $K_{q} \ni \tau \mapsto \ol{\tau} \in K_{q}$. The oriented simplex $\ol{\tau}$ for $\tau \in K_{q}$
is said to have the orientation opposite to $\tau$.
It should be noted that, in usual homology theory, $K_{q}$ is used to define the chain complex of the
simplicial complex $\kcal$.

Noting $\scal_{q}=\wh{K}_{q}/\mf{S}_{q+1}$, we denote the simplex in $\scal_{q}$
corresponding to $\tau \in \wh{K}_{q}$ by $[\tau]$.
For any $[\tau]$ and $[\sigma]$ in $\scal_{q}$, we say they are {\it up neighbors} if they are contained (as a face)
in a common simplex of dimension $q+1$,
and we say they are {\it down neighbors} if they share one common ($q-1$)-dimensional simplex (as a face).
Let $\sigma=\ispa{a_{0}a_{1}\cdots a_{q}} \in K_{q}$ and let $\tau \in K_{q-1}$ such that $[\tau]$ is a face of $[\sigma]$.
Then, the signature $\sgn(\sigma,\tau)$ is defined as $(-1)^{j}$ if $\tau=\ispa{a_{0} \cdots \wh{a_{j}} \cdots a_{q}}$.
When $[\tau]$ is not a face of $[\sigma]$, we put $\sgn(\sigma,\tau)=0$.
It holds that $\sgn(\sigma,\ol{\tau})=\sgn(\ol{\sigma},\tau)=-\sgn(\sigma,\tau)$.
The presentation of orientation of simplices in $\kcal$ might not be so common.
However, it will be useful for example in Section $\ref{ORI}$.

\subsection{Assumption on simplicial complexes.} \label{AS}

Throughout the paper, the simplicial complex $\kcal=(V,\scal)$ is assumed to satisfy all of the following properties.
\begin{itemize}
\item $\kcal$ has bounded degree, namely, there exists a constant $K>0$ such that
for each $F \in \scal$, the number of elements in $\scal$ containing $F$ is not greater than $K$.
\item $\kcal$ has a finite dimension, in the sense that the maximum of the dimensions of simplices in $\scal$ is finite.
We denote by $\dim \kcal$ the maximum of dimensions of simplices in $\scal$.
\item $\kcal$ is pure, in the sense that for any $F \in \scal$, there exists a $G \in \scal$ such that $\dim G=\dim \kcal$
and $F$ is contained in $G$.
\item $\kcal$ is strongly connected,
in the sense that, for two given simplices $\sigma_{1},\sigma_{2} \in \scal_{\dim \kcal}$, there exists a sequence
$\tau_{1},\ldots,\tau_{n} \in \scal_{\dim \kcal}$ such that $\tau_{1}=\sigma_{1}$, $\tau_{n}=\sigma_{2}$ and
$\tau_{i} \cap \tau_{i+1} \in \scal_{\dim \kcal -1}$ for each $i=1,\ldots,n-1$.
\end{itemize}

\subsection{Orientation} \label{SORI}

Let $\kcal=(V,\scal)$ be a simplicial complex (not necessarily satisfy the above assumptions).
Let $\sigma \in K_{q}$, $\tau \in K_{q-1}$. Suppose that $[\tau] \subset [\sigma]$.
Then the orientation of $\tau$ (as an oriented simplex) is said to be {\it induced} by the orientation of $\sigma$ if
$\sgn(\sigma,\tau)=1$.
Let $\sigma_{1},\sigma_{2} \in K_{q}$. Suppose that $[\tau]=[\sigma_{1}] \cap [\sigma_{2}] \in \scal_{q-1}$,
namely suppose that $\sigma_{1}$ and $\sigma_{2}$ are down neighbors.
Then the orientation of $\sigma_{1}$ and $\sigma_{2}$ (as oriented simplices) is said to be {\it coherent} if
$\sgn(\sigma_{1},\tau)\sgn(\sigma_{2},\tau)=-1$.
For an $n$-dimensional simplicial complex $\kcal$, its {\it orientation} means a subset $K_{n}^{o}$ of $K_{n}$
such that $K_{n}^{o} \cap \ol{K_{n}^{o}}=\emptyset$, where $\ol{K_{n}^{o}}=\{\ol{\sigma} \,;\, \sigma \in K_{n}^{o}\}$,
and $K_{n}=K_{n}^{o} \cup \ol{K_{n}^{o}}$.

An $n$-dimensional pure simplicial complex $\kcal$ is said to be {\it coherently orientable}
if there exists an orientation $K_{n}^{o}$ of $K_{n}$ such that the orientation of any two simplices
in $K_{n}^{o}$ which are down neighbors is coherent.
Opposed to this notion, the simplicial complex $\kcal$ is said to be
{\it totally non-coherently orientable} if there exists an orientation $K_{n}^{o}$ such that $\sgn(\sigma_{1},\tau)\sgn(\sigma_{2},\tau)=1$
for any down neighbors $\sigma_{1},\sigma_{2} \in K_{n}^{o}$ where $\tau$ is the common $(n-1)$-face of $\sigma_{1}$ and $\sigma_{2}$.
It seems that the total non-coherent orientability is not commonly used notion.
However, this can be seen in Theorem 7.3 in \cite{HJa}.

\subsection{Function spaces}\label{FSPACE}

One of our purpose is to introduce Grover walks on graphs which are naturally defined by
using combinatorial structures of a simplicial complex and compare its properties with
other operators such as quantum walks (certain unitary operators) based on $\wh{K}_{q}$ defined in \cite{MOS}
and the combinatorial Laplacians discussed in \cite{HJa}. These are defined on different function spaces.
Thus we need to prepare these function spaces and mention about relationships among them.

For any countable set $X$ and functions $f,g:X \to \mb{C}$, we define
\[
\ispa{f,g}_{X}=\sum_{x \in X} f(x) \ol{g(x)},\quad \|f\|_{X}^{2}=\ispa{f,f}_{X}
\]
if they converge. Then the $\ell^{2}$-space $\ell^{2}(X)$ is defined as
\[
\ell^{2}(X)=\{f:X \to \mb{C} \,;\, \|f\|_{X}<+\infty\}.
\]
By the assumption $\ref{AS}$ for our simplicial complex $\kcal=(V,\scal)$,
$\wh{K}_{q}$ is a countable set for any $q$.
Thus, the $\ell^{2}$-spaces $\ell^{2}(\wh{K}_{q})$, $\ell^{2}(K_{q})$ are defined.
We remark that $\ell^{2}(K_{q})$ can be naturally regarded as a subspace of $\ell^{2}(\wh{K}_{q})$.
Indeed, we have an identification
\[
\ell^{2}(K_{q}) \cong \{
f \in \ell^{2}(\wh{K}_{q}) \,;\,
f(s^{\pi})=f(s) \ (\pi \in \mf{A}_{q+1})
\}
\]
Note that, as will be explained in the next section, unitary operators called
simplicial quantum walks introduced in \cite{MOS} is defined on $\ell^{2}(\wh{K}_{q})$
and the combinatorial Laplacian introduced in \cite{HJa} is defined on a subspace of $\ell^{2}(K_{q})$,
the cochain groups.
The chain group $C_{q}(\kcal)$ for a finite simplicial complex $\kcal$ is defined as a
quotient group of a free abelian group with basis $K_{q}$ by the relations $\tau +\ol{\tau}=0$ for each $\tau \in K_{q}$.
It turns out that the chain group $C_{q}(\kcal)$ is also free abelian group with basis $K_{q}^{o}$, a fixed orientation of $K_{q}$.
The cochain group
is then defined as the set of homomorphisms from $C_{q}(\kcal)$ to $\mb{Z}$.
In our case the coefficients is complex numbers and
thus, with our notation, the cochain group with complex coefficients is defined by
\[
C^{q}(\kcal,\mb{C}) = \{f \in \ell^{2}(K_{q}) \,;\, f(\ol{\tau})=-f(\tau) \ (\tau \in K_{q})\}.
\]
We note that in the above the simplicial complex $\kcal$ is not necessarily finite.
Then $C^{q}(\kcal,\mb{C})$ can be regarded as a subspace of $\ell^{2}(\wh{K}_{q})$ as
\[
C^{q}(\kcal,\mb{C}) \cong \{f \in \ell^{2}(\wh{K}_{q}) \,;\, f(s^{\pi})=\sgn (\pi) f(s) \ \ (s \in \wh{K}_{q})\}.
\]
The inner product on $C^{q}(\kcal,\mb{C})$ as a subspace of $\ell^{2}(K_{q})$ is twice the usual inner
product on the cochain group, for example used in \cite{HJa}.
In this context it would be natural to define the space of symmetric functions
\[
C_{+}^{q}(\kcal,\mb{C})=\{f \in \ell^{2}(K_{q}) \,;\, f(\ol{\tau})=f(\tau) \ (\tau \in K_{q})\}
\cong \{f \in \ell^{2}(\wh{K}_{q}) \,;\, f(s^{\pi})= f(s) \ \ (s \in \wh{K}_{q})\}.
\]
Then we have the orthogonal decomposition
\[
\ell^{2}(K_{q})=C^{q}(\kcal,\mb{C}) \oplus C^{q}_{+}(\kcal,\mb{C}).
\]
We remark that $C_{+}^{q}(\kcal,\mb{C})$ is naturally identified with $\ell^{2}(\scal_{q})$ but the inner product
is twice that of $\ell^{2}(\scal_{q})$.

\section{Up and down graphs and Grover walks on them}\label{DEFGR}

In this section, we define our main objects, Grover walks on up and down graphs for
simplicial complex $\kcal=(V,\scal)$. Before giving it, let us review a definition of a
unitary operator discussed in \cite{MOS}.

\subsection{Modified version of an S-quantum walk} \label{MOSS}

We define $\wh{\alpha}_{q-1}:\ell^{2}(\wh{K}_{q-1}) \to \ell^{2}(\wh{K}_{q})$ by
\[
(\wh{\alpha}_{q-1}f)(s)=\frac{1}{\sqrt{\deg (\nu_{q}(s))}}f(\nu_{q}(s)),
\]
where
\[
\nu_{q}:\wh{K}_{q} \to \wh{K}_{q-1},\quad \nu_{q}(a_{0}\cdots a_{q-1}a_{q})=(a_{0} \cdots a_{q-1}),
\]
and, for $t \in \wh{K}_{q-1}$, $\deg (t)$ is defined as
\[
\deg (t)= \sharp (\wh{K}_{q})_{t},\quad
(\wh{K}_{q})_{t}=\{s \in \wh{K}_{q} \,;\,
\nu_{q}(s)=t\}.
\]
Then, the adjoint operator $\wh{\alpha}_{q-1}^{*}:\ell^{2}(\wh{K}_{q}) \to \ell^{2}(\wh{K}_{q-1})$
is given by
\[
(\wh{\alpha}_{q-1}^{*}g)(t)= \frac{1}{\sqrt{\deg (t)}} \sum_{s \in (\wh{K}_{q})_{t}}g(s).
\]
We have $\wh{\alpha}_{q-1}^{*}\wh{\alpha}_{q-1}=I$ on $\ell^{2}(\wh{K}_{q-1})$,
and thus $\wh{\alpha}_{q-1}\wh{\alpha}_{q-1}^{*}$ is a projection on $\ell^{2}(\wh{K}_{q})$.
Therefore, the operator $C_{q}:\ell^{2}(\wh{K}_{q}) \to \ell^{2}(\wh{K}_{q})$ defined as
\begin{equation}\label{coin1}
C_{q}=2\wh{\alpha}_{q-1}\wh{\alpha}_{q-1}^{*} -I
\end{equation}
is a unitary operator satisfying $C_{q}^{2}=I$. The operator $C_{q}$ in $\eqref{coin1}$ is used in \cite{MOS} as a `coin' operator
to define quantum walks, called S-quantum walks. To define an operator closely related to S-quantum walks,
one need to prepare a `shift' operator.
We use the projection
$\wh{P}_{q}:\ell^{2}(\wh{K}_{q}) \to \ell^{2}(\wh{K}_{q})$ defined by
\[
(\wh{P}_{q}f)(s)=\frac{1}{(q+1)!} \sum_{\pi \in \mf{S}_{q+1}} \sgn(\pi) f(s^{\pi}) \quad
(f \in \ell^{2}(\wh{K}_{q})).
\]
We note that $\im(\wh{P}_{q})=C^{q}(\kcal,\mb{C})$ which is the cochain group of dimension $q$.
\begin{defin}\label{sgrod}
We define a unitary operator $G_{q}:\ell^{2}(\wh{K}_{q}) \to \ell^{2}(\wh{K}_{q})$ by the formula
\[
G_{q}=S_{q}C_{q},\quad S_{q}=2\wh{P}_{q}-I.
\]
The corresponding discriminant operator $D(G_{q}):\ell^{2}(\wh{K}_{q-1}) \to \ell^{2}(\wh{K}_{q-1})$ is
defined by
\[
D(G_{q})=\wh{\alpha}_{q-1}^{*}G_{q} \wh{\alpha}_{q-1}=\wh{\alpha}_{q-1}^{*}S_{q} \wh{\alpha}_{q-1}.
\]
\end{defin}

\begin{rem}
It should be remarked that for $\pi \in \mf{S}_{q+1}$ the operator $S^{\pi}:\ell^{2}(\wh{K}_{q}) \to \ell^{2}(\wh{K}_{q})$
defined by $(S^{\pi}f)(s)=f(s^{\pi})$ is used in \cite{MOS} instead of our $S_{q}$,
and in this case $(S^{\pi})^{2}$ need not to equal the identity.
It seems that one single choice of $\pi \in \mf{S}_{q+1}$ might not be enough to relate it with geometry.
Our shift operator $S_{q}$ defined above is to relate the combinatorial Laplacian. See Section $\ref{SIMP}$ below.
\end{rem}

\subsection{Up and down graphs and their Grover walks} \label{UPDOWN}

Let us turn to give the definitions of our main objects.
\begin{defin}
For any non-negative integer $q$ with $0 \leq q \leq \dim \kcal -1$,
the up graph $X_{q}=(V(X_{q}),E(X_{q}))$ of the simplicial complex $\kcal$,
where $V(X_{q})$ is the set of vertices and $E(X_{q})$ is the set of oriented edges, is defined as follows.
\[
\begin{split}
V(X_{q}) & = K_{q},\\
E(X_{q}) & = \{(\tau_{1},\tau_{2}) \in K_{q} \times K_{q} \,;\,
\tau_{1} \neq \tau_{2},\,\tau_{1} \neq \ol{\tau_{2}},\,\mbox{$[\tau_{1}]$ and $[\tau_{2}]$ are up neighbors}\}.
\end{split}
\]
For any non-negative integer $q$ with $1 \leq q \leq \dim \kcal$,
the down graph $Y_{q}=(V(Y_{q}), E(Y_{q}))$ of the simplicial complex $\kcal$
is defined as follows.
\[
\begin{split}
V(Y_{q}) & = K_{q},\\
E(Y_{q}) & = \{(\tau_{1},\tau_{2}) \in K_{q} \times K_{q} \,;\,
\tau_{1} \neq \tau_{2},\,\tau_{1} \neq \ol{\tau_{2}},\,\mbox{$[\tau_{1}]$ and $[\tau_{2}]$ are down neighbors}\}.
\end{split}
\]
\end{defin}
The down graph is essential the same as a dual graph discussed in \cite{HJa}.
We remark that up and down graphs $X_{q}$ and $Y_{q}$ have `redundant' edges.
Namely, if $(\sigma,\tau) \in E(Y_{q})$ then $(\sigma,\ol{\tau})$ is also an edge in $Y_{q}$.
This redundancy will be necessary to relates the operators each other.
But, in the actual computation, it would be reasonable to reduce this redundancy.
To do it, we fix an orientation $K_{q}^{o}$ of $K_{q}$ and we define the {\it reduced down graph} $Y_{q}^{r}$ by
\[
\begin{split}
V(Y_{q}^{r}) & = K_{q}^{o},\\
E(Y_{q}) & = \{(\tau_{1},\tau_{2}) \in K_{q}^{o} \times K_{q}^{o} \,;\,
\tau_{1} \neq \tau_{2},\,\tau_{1} \neq \ol{\tau_{2}},\,\mbox{$[\tau_{1}]$ and $[\tau_{2}]$ are down neighbors}\}.
\end{split}
\]
We define the {\it reduced up graph} $X_{q}^{r}$ by a similar fashion.

For any $\tau \in K_{q}$, we set
\[
\begin{split}
\deg_{X}(\tau) & =\sharp \{[\sigma] \in \scal_{q+1}\,;\, \mbox{$[\sigma]$ contains $[\tau]$ as a face}\},\\
\deg_{Y}(\tau) & =\sharp \{[\tau'] \in \scal_{q} \,;\, \mbox{$[\tau]$ and $[\tau']$ are down neighbors}\}. \\
\end{split}
\]
For a given ordered simplex $s \in \wh{K}_{q}$, we have $\deg_{X}(\ispa{s})=\deg(s)$.
For each $\tau \in K_{q}$ the degree of $\tau$ as a vertex of the graph $X_{q}$ is $2(q+1)\deg_{X}(\tau)$
and the degree of $\tau$ as a vertex of the graph $Y_{q}$ is $2\deg_{Y}(\tau)$.
If $\tau_{1},\tau_{2} \in K_{q}$ are up neighbors, then they are also a
down neighbors. Thus $E(X_{q})$ is naturally regarded as a subset of $E(Y_{q})$.

Before giving the definition of up and down Grover walk,
we need to prepare a property of signature.
Suppose that $[\tau_{1}], [\tau_{2}] \in \scal_{q}$ are up neighbors. Then there is a unique $[\sigma_{q+1}(\tau_{1},\tau_{2})] \in \scal_{q+1}$
such that $[\tau_{1}]$ and $[\tau_{2}]$ are faces of $[\sigma_{q+1}(\tau_{1},\tau_{2})]$.
Although there are two orientation of such an $[\sigma_{q+1}(\tau_{1},\tau_{2})]$, we have
\begin{equation}\label{sgn1}
\sgn(\sigma_{q+1}(\tau_{1},\tau_{2}),\tau_{1})\sgn(\sigma_{q+1}(\tau_{1},\tau_{2}),\tau_{2})
=\sgn(\ol{\sigma_{q+1}(\tau_{1},\tau_{2})},\tau_{1})\sgn(\ol{\sigma_{q+1}(\tau_{1},\tau_{2})},\tau_{2}).
\end{equation}
Likewise, For any down neighbors $[\tau_{1}], [\tau_{2}] \in \scal_{q}$,
the simplex $[\sigma_{q-1}(\tau_{1},\tau_{2})]=[\tau_{1}] \cap [\tau_{2}]$ have two orientation.
However, we have
\begin{equation}\label{sgn2}
\sgn(\tau_{1},\sigma_{q-1}(\tau_{1},\tau_{2}))\sgn(\tau_{2},\sigma_{q-1}(\tau_{1},\tau_{2}))
=\sgn(\tau_{1},\ol{\sigma_{q-1}(\tau_{1},\tau_{2})})\sgn(\tau_{2},\ol{\sigma_{q-1}(\tau_{1},\tau_{2})})
\end{equation}
The equations $\eqref{sgn1}$, $\eqref{sgn2}$
mean that the quantities in the left-hand sides of $\eqref{sgn1}$ and $\eqref{sgn2}$ do not depend on
the choice of the orientation of $[\sigma_{q+1}(\tau_{1},\tau_{2})]$ and $[\sigma_{q-1}(\tau_{1},\tau_{2})]$.
These can be deduced from a simple property of the signature.

\begin{defin}
\noindent{{\rm (1)}} \hspace{2pt}
We define operators $d_{X_{q}}:\ell^{2}(E(X_{q})) \to \ell^{2}(K_{q})$, $d_{Y_{q}}:\ell^{2}(E(Y_{q})) \to \ell^{2}(K_{q})$ by
the following formula.
\[
\begin{split}
(d_{X_{q}}g)(\tau) & = \frac{1}{\sqrt{2(q+1)\deg_{X}(\tau)}} \sum_{\tau' \in K_{q}\,;\, (\tau,\tau') \in E(X_{q})}
g(\tau,\tau') \quad (g \in \ell^{2}(E(X_{q}))),\\
(d_{Y_{q}}g)(\tau) & = \frac{1}{\sqrt{2\deg_{Y}(\tau)}} \sum_{\tau' \in K_{q}\,;\, (\tau,\tau') \in E(Y_{q})}
g(\tau,\tau') \quad (g \in \ell^{2}(E(Y_{q}))).
\end{split}
\]
\noindent{{\rm (2)}} \hspace{2pt}
The shift operators $S^{\up}:\ell^{2}(E(X_{q})) \to \ell^{2}(E(X_{q}))$ and $S^{\down}:\ell^{2}(E(Y_{q})) \to \ell^{2}(E(Y_{q}))$ are defined as
\[
\begin{split}
(S^{\up}g)(\tau_{1},\tau_{2}) & = \eta^{\up}(\tau_{1},\tau_{2}) g(\tau_{2},\tau_{1}) \quad (g \in \ell^{2}(E(X_{q}))),\\
(S^{\down}g)(\tau_{1},\tau_{2}) & = \eta^{\down}(\tau_{1},\tau_{2}) g(\tau_{2},\tau_{1}) \quad (g \in \ell^{2}(E(Y_{q}))),
\end{split}
\]
where the functions $\eta^{\up}$ on $E(X_{q})$ and $\eta^{\down}$ on $E(Y_{q})$ are defined as
\[
\begin{split}
\eta^{\up}(\tau_{1},\tau_{2}) & =\sgn(\sigma_{q+1}(\tau_{1},\tau_{2}),\tau_{1})
\sgn(\sigma_{q+1}(\tau_{1},\tau_{2}),\tau_{2}) \quad ((\tau_{1},\tau_{2}) \in E(X_{q})), \\
\eta^{\down}(\tau_{1},\tau_{2}) & = \sgn(\tau_{1},\sigma_{q-1}(\tau_{1},\tau_{2}))
\sgn(\tau_{2},\sigma_{q-1}(\tau_{1},\tau_{2})) \quad ((\tau_{1},\tau_{2}) \in E(Y_{q})).
\end{split}
\]
\noindent{{\rm (3)}} \hspace{2pt}
The up-Grover walk $U_{q}^{\up}:\ell^{2}(E(X_{q})) \to \ell^{2}(E(X_{q}))$
and the down-Grover walk $U_{q}^{\down}:\ell^{2}(E(Y_{q})) \to \ell^{2}(E(Y_{q}))$ are defined as
\[
\begin{split}
U_{q}^{\up} & = S^{\up} (2d_{X_{q}}^{*}d_{X_{q}} -I), \\
U_{q}^{\down} & = S^{\down} (2d_{Y_{q}}^{*}d_{Y_{q}}-I).
\end{split}
\]
\noindent{{\rm (4)}} The discriminant operators $D_{q}^{\up}:=D(U_{q}^{\up}):\ell^{2}(K_{q}) \to \ell^{2}(K_{q})$,
$D_{q}^{\down}:=D(U_{q}^{\down}):\ell^{2}(K_{q}) \to \ell^{2}(K_{q})$ are given by the following formula.
\[
D_{q}^{\up}=d_{X_{q}} S^{\up} d_{X_{q}}^{*},\quad
D_{q}^{\down}=d_{Y_{q}} S^{\down} d_{Y_{q}}^{*}.
\]
\end{defin}
In the following some remarks and simple properties deduced from the definitions are listed.
\begin{enumerate}
\item The operators $d_{X_{q}}$ and $d_{Y_{q}}$ are bounded operators whose operator norms are not greater than $1$.
\item The operators $d_{X_{q}}^{*}$ and $d_{Y_{q}}^{*}$ are
adjoint operators whose concrete forms are given by
\begin{equation}\label{dadj}
\begin{split}
(d_{X_{q}}^{*}f)(\tau_{1},\tau_{2}) & = \frac{1}{\sqrt{2(q+1)\deg_{X}(\tau_{1})}} f(\tau_{1}) \quad (f \in \ell^{2}(K_{q})), \\
(d_{Y_{q}}^{*}f) (\tau_{1},\tau_{2}) & = \frac{1}{\sqrt{2\deg_{Y}(\tau_{1})}} f(\tau_{1}) \quad (f \in \ell^{2}(K_{q})).
\end{split}
\end{equation}
These satisfies $d_{X_{q}}d_{X_{q}}^{*}=I$, $d_{Y_{q}} d_{Y_{q}}^{*}=I$ on $\ell^{2}(K_{q})$.
Thus, $d_{X_{q}}^{*}d_{X_{q}}$ and $d_{Y_{q}}^{*}d_{Y_{q}}$ are projections,
and hence $2d_{X_{q}}^{*}d_{X_{q}} -I$ and $2d_{Y_{q}}^{*}d_{Y_{q}}-I$ are unitary operators.
\item It is straightforward to check that $S^{\up}$ and $S^{\down}$ are unitary operators and
they satisfy $(S^{\up})^{2}=I$, $(S^{\down})^{2}=I$.
Therefore, $U_{q}^{\up}$ and $U_{q}^{\down}$ are regarded as an abstract Szegedy walk (\cite{SS}).
The definition of these unitary operators comes from, essentially, the description of
unitary operators in \cite{HKSS}, except one point that we adjust the definition of the shift operators
to take the orientation of simplices into account.
\item It would be useful to give concrete forms of the discriminant operators
which, for $g \in \ell^{2}(K_{q})$, are given as
\begin{equation}\label{disc22}
\begin{split}
(D_{q}^{\up}g)(\tau) & =\sum_{\tau' \in K_{q}\,;\,(\tau,\tau') \in E(X_{q})}
\frac{1}{\sqrt{2(q+1)\deg_{X}(\tau)} \sqrt{2(q+1)\deg_{X}(\tau')}}
\eta^{\up}(\tau,\tau')g(\tau'), \\
(D_{q}^{\down}g)(\tau) & = \sum_{\tau' \in K_{q}\,;\,(\tau,\tau') \in E(Y_{q})}
\frac{1}{\sqrt{2\deg_{Y}(\tau)} \sqrt{2\deg_{Y}(\tau')}}
\eta^{\down}(\tau,\tau')g(\tau').
\end{split}
\end{equation}
\item It would be reasonable to note that the function $\eta_{q}^{\up}$
satisfies
\begin{equation}\label{sgneta}
\eta_{q}^{\up}(\ol{\tau_{1}},\tau_{2})=\eta_{q}^{\up}(\tau_{1},\ol{\tau_{2}})=-\eta_{q}^{\up}(\tau_{1},\tau_{2}),
\end{equation}
where $(\tau_{1},\tau_{2}) \in E(X_{q})$, and similar formula also holds for $\eta_{q}^{\down}$.
\end{enumerate}

\section{Fundamental properties of up and down Grover walks} \label{SIMP}

In this section, we shall investigate some fundamental properties of unitary operators $G_{q}$, $U_{q}^{\up}$, $U_{q}^{\down}$.
Since the spectral structures of these unitary operators are almost determined by their discriminant operators,
we mainly investigate properties of their discriminant operators.

\subsection{Another description for discriminants} \label{ANOT}

First of all, we show that the discriminants $D_{q}^{\up}$, $D_{q}^{\down}$ are essentially defined on
the cochain group, $C^{q}(\kcal,\mb{C})$.
\begin{thm}\label{ker1}
Let $P:\ell^{2}(K_{q}) \to \ell^{2}(K_{q})$ be the projection onto $C^{q}(\kcal,\mb{C})$. Then, we have
\[
D_{q}^{\up}=PD_{q}^{\up}P,\quad D_{q}^{\down}=PD_{q}^{\down}P.
\]
In particular, $C^{q}(\kcal,\mb{C})$ is invariant under $D_{q}^{\up}$ and $D_{q}^{\down}$,
and $D_{q}^{\up}C_{+}^{q}(\kcal,\mb{C})=D_{q}^{\down}C^{q}_{+}(\kcal,\mb{C})=0$.
\end{thm}
\begin{proof}
The projection $P$ onto $C^{q}(\kcal,\mb{C})$ is defined by $\dsp Pg(\tau)=\frac{1}{2}(g(\tau)-g(\ol{\tau}))$.
If $(\tau,\tau') \in E(X_{q})$ then $(\tau,\ol{\tau'})$ is also in $E(X_{q})$. Thus, equations $\eqref{disc22}$ and $\eqref{sgneta}$
show the proposition.
\end{proof}
\begin{cor}
The up and down Grover walks $U_{q}^{\up}$, $U_{q}^{\down}$ have always eigenvalue $\pm i$.
\end{cor}
\begin{proof}
Let us denote $\spec(A)$ and $\spec_{p}(A)$ the spectrum and the
set of eigenvalues of an operator $A$, respectively. Then, it is well-known (\cite{SS})
that if $t \in \spec(D_{q}^{\up})$ then $t \pm i\sqrt{1-t^{2}} \in \spec(U_{q}^{\up})$,
and if $t \in \spec_{p}(D_{q}^{\up})$ then $t\pm i\sqrt{1-t^{2}} \in \spec_{p}(U_{q}^{\up})$.
\end{proof}

By Proposition $\ref{ker1}$, it turns out that we only need to consider the discriminant operators on $C^{q}(\kcal,\mb{C})$.
The discriminant operators have a nice representation which
are shown in the following proposition.
\begin{prop}\label{ano1}
We define operators $a_{q}, b_{q}:C^{q}(\kcal,\mb{C}) \to C^{q+1}(\kcal,\mb{C})$ by
\[
\begin{split}
(a_{q}f)(\sigma) & = \frac{1}{2}\sum_{\tau \in K_{q}} \frac{1}{\sqrt{\deg_{X}(\tau)}}\sgn(\sigma,\tau) f(\tau),\\
(b_{q}f)(\sigma) & = \frac{1}{2}\sum_{\tau \in K_{q}} \frac{1}{\sqrt{\deg_{Y}(\sigma)}} \sgn(\sigma,\tau) f(\tau)
\end{split}
\]
for $f \in C^{q}(\kcal,\mb{C})$ and $\sigma \in K_{q+1}$.
We also define an operator $A_{q}^{\down}:\ell^{2}(K_{q}) \to \ell^{2}(K_{q})$ by $\dsp (A_{q}^{\down}f)(\tau)=\frac{1}{\sqrt{\deg_{Y}(\tau)}}f(\tau)$.
Then, on the subspace $C^{q}(\kcal,\mb{C})$, we have the following formulas.
\[
\begin{split}
D_{q}^{\up} & = \frac{1}{q+1} (a_{q}^{*}a_{q}-I),\\
D_{q}^{\down} & =  b_{q-1}b_{q-1}^{*}-(q+1)(A_{q}^{\down})^{2}.
\end{split}
\]
\end{prop}
\begin{proof}
The adjoint operators of $a_{q}$ and $b_{q-1}$ are given by
\[
(a_{q}^{*}g) (\tau) = \frac{1}{2}
\sum_{\sigma \in K_{q+1}} \frac{\sgn(\sigma,\tau)}{\sqrt{\deg_{X}(\tau)}} g(\sigma), \quad
(b_{q-1}^{*}f) (\mu) = \frac{1}{2}
\sum_{\tau' \in K_{q}} \frac{\sgn(\tau',\mu)}{\sqrt{\deg_{Y}(\tau')}} f(\tau'),
\]
where $g \in C^{q+1}(\kcal,\mb{C})$, $f \in C^{q}(\kcal,\mb{C})$, $\tau \in K_{q}$ and $\mu \in K_{q-1}$.
For $\tau,\tau' \in K_{q}$, we have
\[
\sum_{\sigma \in K_{q+1}} \sgn(\sigma,\tau)\sgn(\sigma,\tau')=
\begin{cases}
0 & \mbox{if $[\tau] \neq [\tau']$ and $[\tau]$ and $[\tau']$ are not up neighbors,} \\
2\eta_{q}^{up}(\tau,\tau') & \mbox{if $[\tau] \neq [\tau']$ and $[\tau]$ and $[\tau']$ are up neighbors,} \\
2\deg_{X}(\tau) & \mbox{if $[\tau] = [\tau']$.}
\end{cases}
\]
Similar property holds for down neighbors. The statement
follows by a direct computation using these formulas.
\end{proof}

\begin{cor}\label{anoC}
Let $\spec(A)$ be the spectrum of an operator $A$. Then the following hold.
\begin{enumerate}
\item Suppose that our simplicial complex $\kcal=(V,\scal)$ is finite.
Suppose further that each $q$-dimensional simplex in $\scal$ is contained in
exactly $L$ $(q+1)$-dimensional simplices. Then we have
\[
\spec\left(
(q+2) +(L-1)(q+2)D_{q+1}^{\down}
\right)
\overset{{\scriptstyle \circ}}{=}
\spec\left(
L+L(q+1)D_{q}^{\up}
\right),
\]
where ${\rm S}(A) \overset{{\scriptstyle \circ}}{=} {\rm S}(B)$ means that the eigenvalues of $A$ and $B$ differ
only in the multiplicities of zero and other eigenvalues are the same with the same multiplicities.
\item For infinite simplicial complex with the same assumtion as in {\rm (1)},
we have
\[
\spec\left(
(q+2) +(L-1)(q+2)D_{q+1}^{\down}
\right)
\overset{{\scriptstyle \circ}}{=}
\spec\left(
L+L(q+1)D_{q}^{\up}
\right),
\]
where ${\rm spec}(A) \overset{{\scriptstyle \circ}}{=} {\rm spec}(B)$ means the spectrum of $A$ and $B$
differ only in zero.
\item Let $\kcal=(V,\scal)$ be the $(n-1)$-dimensional simplex.
Then we have
\[
\spec\left(
(q+2) +(n-q-2)(q+2)D_{q+1}^{\down}
\right)
\overset{{\scriptstyle \circ}}{=}
\spec\left(
(n-q-1)+(n-q-1)(q+1)D_{q}^{\up}
\right),
\]
\end{enumerate}
\end{cor}

\begin{proof}
The assertion follows from Proposition $\ref{ano1}$ and the equation $(2.6)$ in \cite{HJa}.
(For infinite simplicial complex, the equation $(2,6)$ in \cite{HJa} still works. See p.180 in \cite{KR}.)
\end{proof}

\subsection{Relation with certain combinatorial Laplacians}\label{LAP}

Proposition $\ref{ano1}$ makes us to discuss a relation between the discriminants $D_{q}^{\up}$, $D_{q}^{\down}$
and the combinatorial Laplacians introduced and investigated in \cite{HJa}.
To introduce the combinatorial Laplacian, we need the coboundary operator $\delta_{q}:C^{q}(\kcal,\mb{C}) \to C^{q+1}(\kcal,\mb{C})$
defined by
\[
(\delta_{q}f)(\sigma)=\frac{1}{2}\sum_{\tau \in K_{q}} \sgn(\sigma,\tau)f(\tau) \quad (f \in C^{q}(\kcal,\mb{C}),\ \sigma \in K_{q+1}).
\]
Then the combinatorial Laplacian $\lcal_{q}:C^{q}(\kcal,\mb{C}) \to C^{q}(\kcal,\mb{C})$ with the weight function $w \equiv 1$,
and the up and down Laplacian $\lcal_{q}^{\up},\lcal_{q}^{\down}:C^{q}(\kcal,\mb{C}) \to C^{q}(\kcal,\mb{C})$ are defined as
\[
\lcal_{q}=\lcal_{q}^{\up} +\lcal_{q}^{\down},\quad
\lcal_{q}^{\up}=\delta_{q}^{*}\delta_{q},\quad \lcal_{q}^{\down}=\delta_{q-1}\delta_{q-1}^{*}.
\]
We then have the following.
\begin{thm}\label{Lap2}
We define an operator $A_{q}^{\up}:\ell^{2}(K_{q}) \to \ell^{2}(K_{q})$ by
$\dsp (A_{q}^{\up}f)(\tau)=\frac{1}{\sqrt{\deg_{X}(\tau)}}f(\tau)$. Then we have the following.
\[
D_{q}^{\up} = \frac{1}{q+1}\left(
A_{q}^{\up}\lcal_{q}^{\up}A_{q}^{\up} -I
\right), \quad
D_{q}^{\down} =
A_{q}^{\down} \left(
\lcal_{q}^{\down}-(q+1)I
\right) A_{q}^{\down},
\]
where the operator $A_{q}^{\down}$ is defined in Proposition $\ref{ano1}$.
\end{thm}
\begin{proof}
The operators $a_{q}$ and $b_{q}$ given in Proposition $\ref{ano1}$ are written in the form
\[
a_{q}=\delta_{q}A_{q}^{\up},\quad
b_{q}=A_{q+1}^{\down}\delta_{q}.
\]
From this and Proposition $\ref{ano1}$, the statement follows.
\end{proof}
When $n=\dim \kcal$, the $n$-th cohomology group with complex coefficient $H^{n}(\kcal,\mb{C})$ is
isomorphic to $\ker(\lcal_{n}^{\down})$, we see

\begin{cor}\label{cohom1}
Suppose that $\kcal$ is finite and $n$-dimensional.
Suppose also that each $[\tau] \in \scal_{n}$ has exactly $n+1$ down neighbors. Then
the eigenspace of $D_{n}^{\down}$ with eigenvalue $-1$ is isomorphic to $H^{n}(\kcal,\mb{C})$.
\end{cor}

\subsection{Relation with S-quantum walks}\label{SGRO}

Next, let us consider a relation between the up Grover walks and modified S-quantum walks defined in $\ref{sgrod}$.
The discriminant operator $D(G_{q})$ of the S-quantum walk $G_{q}$ is defined also in $\ref{sgrod}$ and
is written explicitly in the following form.
\begin{equation}\label{sgroE}
[D(G_{q})f] (t)=-f(t)+\frac{2}{(q+1)!\sqrt{\deg(t)}} \sum_{s \in (\wh{K}_{q})_{t}} \sum_{\pi \in \mf{S}_{q+1}}
\frac{\sgn(\pi)}{\sqrt{\deg(\nu_{q}(s^{\pi}))}} f(\nu_{q}(s^{\pi})),
\end{equation}
where $f \in \ell^{2}(\wh{K_{q-1}})$, $t \in \wh{K_{q-1}}$.
\begin{thm}\label{sgroT}
Let us identify $C^{q-1}(\kcal,\mb{C})$, $C^{q-1}_{+}(\kcal,\mb{C})$ as subspaces in $\ell^{2}(\wh{K}_{q-1})$
as in $\ref{FSPACE}$. Then we have the following.
\begin{enumerate}
\item For $f \in C_{+}^{q-1}(\kcal,\mb{C})$, we have $D(G_{q})f=-f$.
\item $C^{q-1}(\kcal,\mb{C})$ is an invariant subspace of $D(G_{q})$.
For $f \in C^{q-1}(\kcal,\mb{C})$, we have
\[
(I-D(G_{q}))f=\frac{2q}{q+1} (I-D_{q-1}^{\up})f.
\]
\end{enumerate}
\end{thm}
\begin{proof}
We identify $\mf{S}_{q}$ with the subgroup in $\mf{S}_{q+1}$ as $\mf{S}_{q}=\{\pi \in \mf{S}_{q+1} \,;\, \pi(q)=q\}$.
Then, For any $\mu \in \mf{S}_{q}$ and $s \in \wh{K}_{q}$, $t \in \wh{K}_{q-1}$,
we see $\nu_{q}(s^{\mu})=\nu_{q}(s)^{\mu}$.
For $t \in \wh{K}_{q-1}$, $s \in (\wh{K}_{q})_{t^{\mu}}$ if and only if $s^{\mu^{-1}} \in (\wh{K}_{q})_{t}$.
In particular, $\deg(t^{\mu})=\deg(t)$.
By using $\eqref{sgroE}$, we have, for any $f \in \ell^{2}(\wh{K}_{q-1})$, $t \in \wh{K}_{q-1}$ and $\mu \in \mf{S}_{q}$,
\[
[D(G_{q})f] (t^{\mu})=-f(t^{\mu}) +
\frac{2\sgn(\mu)}{(q+1)! \sqrt{\deg(t)}}
\sum_{s \in (\wh{K}_{q})_{t}} \sum_{\pi \in \mf{S}_{q+1}}
\frac{\sgn(\pi)f(\nu_{q}(s^{\pi}))}{\sqrt{\deg(\nu_{q}(s^{\pi}))}}
\]
From this it is clear that $C^{q-1}(\kcal,\mb{C})$ is an invariant subspace of $D(G_{q})$.

For any $j=0,1,\ldots,q$, define $\pi_{j} \in \mf{S}_{q+1}$ by
\[
\pi_{j}(l)=
\begin{cases}
l & (0 \leq l \leq j-1), \\
l+1 & (j \leq l \leq q-1), \\
j & (j=q).
\end{cases}
\]
Then we have a left coset decomposition $\mf{S}_{q+1}=\bigsqcup_{j=0}^{q}\pi_{j}\mf{S}_{q}$.
Thus, for any $f \in \ell^{2}(\wh{K}_{q-1})$ and $t \in \wh{K}_{q-1}$, we have
\begin{equation}\label{sym2}
[D(G_{q})f](t) = -f(t) + \frac{2}{(q+1)! \sqrt{\deg(t)}}\sum_{s \in (\wh{K}_{q})_{t}} \sum_{j=0}^{q} (-1)^{q-j}
\sum_{\mu \in \mf{S}_{q}} \frac{\sgn(\mu)}{\sqrt{\deg(\nu_{q}(s^{\pi_{j}}))}}
f(\nu_{q}(s^{\pi_{j}})^{\mu})
\end{equation}
If $f \in C_{+}^{q-1}(\kcal,\mb{C})$, which means $f$ is invariant under the action of $\mf{S}_{q}$,
the last term in $\eqref{sym2}$ vanishes due to the sum over all $\mu \in \mf{S}_{q}$, and hence we have $D(G_{q})f=-f$.
Now let $f \in C^{q-1}(\kcal,\mb{C})$. Then in $\eqref{sym2}$, we see $\sgn(\mu)f(\nu_{q}(s^{\pi_{j}})^{\mu})=f(\nu_{q}(s^{\pi_{j}}))$.
Thus, the summation does not depend on $\mu \in \mf{S}_{q}$.
The term $j=q$ in the sum and the first term give $-\frac{q-1}{q+1}f(t)$.
Thus, $\eqref{sym2}$ becomes
\[
[D(G_{q})f](t) = - \frac{q-1}{q+1} f(t) + \frac{2}{(q+1) \sqrt{\deg(t)}}\sum_{s \in (\wh{K}_{q})_{t}} \sum_{j=0}^{q-1} (-1)^{q-j}
 \frac{f(\nu_{q}(s^{\pi_{j}}))}{\sqrt{\deg(\nu_{q}(s^{\pi_{j}}))}}
\]
In the above, twise the summation in $j=0,\ldots,q-1$ and $s \in (\wh{K}_{q})_{t}$ in the above
is equivalent to the summation over all edges of $E(X_{q-1})$ with origin $\ispa{t}$ and
the terminus $\ispa{\nu_{q}(s^{\pi_{j}})}$. We also have
\[
\sgn(\ispa{s},\ispa{t})=(-1)^{q},\quad \sgn(\ispa{s},\ispa{\nu_{q}(s^{\pi_{j}})})=(-1)^{j}
\]
for all $s \in (\wh{K}_{q})_{t}$. From this the statement follows.
\end{proof}

Note that we have $-I \leq D \leq I$ for $D=D(G_{q})$ or $D=D_{q-1}^{\up}$.
Therefore, we have the following.
\begin{cor}\label{NO}
Let $q \geq 2$. Then $D_{q-1}^{\up}$ does not have eigenvalue $-1$.
\end{cor}

\section{Spectrum and combinatorial properties}\label{ORI}

It seems that the up and down Grover walks, or strictly speaking their discriminants,
have much information on geometry of underlying simplicial complex.
One of evidences is Theorem $\ref{Lap2}$ on the relation between them and the combinatorial
Laplacian, because Laplacian has much geometric information.
Another evidence will be the following.
\begin{thm}\label{ori1}
Assume that our simplicial complex $\kcal$ is finite and satisfy all the assumtion in Subsection $\ref{AS}$,
and let $n=\dim \kcal$.
Then, the following holds.
\begin{enumerate}
\item The $n$-down discriminant $D_{n}^{\down}$ has eigenvalue $1$ if and only if
$\kcal$ is totally non-coherently orientable.
\item $D_{n}^{\down}$ has eigenvalue $-1$ if and only if $\kcal$ is coherently orientable.
\end{enumerate}
\end{thm}

We remark that similar statements in Theorem $\ref{ori1}$ holds also for
$D_{q}^{\down}$ with $q<n$, but for this case, $\kcal$ should be replaced by its $q$-skeleton.

To prove Theorem $\ref{ori1}$, we start with the following lemma.
\begin{lem}\label{oriL}
Let $q=1,\ldots,n$. Let $g \in \ell^{2}(K_{q})$. Then $g$ is an eigenfunction of $D_{q}^{\down}$ with
eigenvalue $1$ $(${\it resp}. eigenvalue $-1$$)$ if and only if we have
\begin{equation}\label{eqE}
S^{\down}d_{Y_{q}}^{*}g=d_{Y_{q}}^{*}g \quad (\mbox{{\it resp.}} \ \
S^{\down}d_{Y_{q}}^{*}g=-d_{Y_{q}}^{*}g).
\end{equation}
\end{lem}
\begin{proof}
If $0 \neq g \in \ell^{2}(K_{q})$ satisfy $\eqref{eqE}$, applying $d_{Y_{q}}$ to $\eqref{eqE}$ shows that
$g$ is an eigenfunction with eigenvalue $\pm 1$.
We set $V=\{\vphi \in \ell^{2}(E(Y_{q})) \,;\, S^{\down}\vphi=\vphi\}$ and $\pi_{V}=\frac{1}{2}(S^{\down}+I)$.
Then $\pi_{V}$ is the projection onto the closed subspace $V$ and $S^{\down}=2\pi_{V}-I$.
We also set $W=d_{Y_{q}}^{*}\ell^{2}(K_{q})$. Then $W$ is also closed due to the relation $d_{Y_{q}}d_{Y_{q}}^{*}=I$.
If $g \in \ell^{2}(K_{q})$ is an eigenfunction of $D_{q}^{\down}$ with eigenvalue $1$, we have
$g=Dg=d_{Y_{q}}S^{\down}d_{Y_{q}}^{*}g$. From this we have $d_{Y_{q}}\pi_{V} d_{Y_{q}}^{*}g=g$.
We set $\phi=d_{Y_{q}}^{*}g$ and $\psi =\pi_{V}d_{Y_{q}}^{*}g$. Then
$d_{Y_{q}}\phi=d_{Y_{q}}d_{Y_{q}}^{*}g=g$ and $d_{Y_{q}}\psi=d_{Y_{q}}\pi_{V}d_{Y_{q}}^{*}g=g$.
Thus, $\phi-\psi \in \ker(d_{Y_{q}})=W^{\perp}$. Since $\pi_{V}^{2}=I$, we also have $\phi-\psi \in \ker(\pi_{V})=V^{\perp}$.
Thus $\phi-\psi \in W^{\perp} \cap V^{\perp}=(W+V)^{\perp}$. By definition we see $\phi \in W$ and $\psi \in V$ and hence
$\phi -\psi \in (W+V)^{\perp} \cap (W+V)=0$. This means $\pi_{V}d_{Y_{q}}^{*}g=d_{Y_{q}}^{*}g$
and hence $S^{\down}d_{Y_{q}}^{*}g=d_{Y_{q}}^{*}g$.

If $g$ is an eigenfunction of $D_{q}^{\down}$ with eigenvalue $-1$, then $d_{Y_{q}} \pi_{V} d_{Y_{q}}^{*}g=0$.
We set $\rho=\pi_{V}d_{Y_{q}}^{*}g$. Then we have $d_{Y_{q}}\rho=0$.
Hence $\rho \in V \cap W^{\perp}=(V^{\perp}+W)^{\perp}$.
But we can write $\rho=d_{Y_{q}}^{*}g -(I-\pi_{V})d_{Y_{q}}^{*}g \in W+V^{\perp}$. Therefore $\rho=0$.
From this we have $S^{\down}d_{Y_{q}}^{*}g=(2\pi_{V}-I)d_{Y_{q}}^{*}g=-d_{Y_{q}}^{*}g$.
\end{proof}

\vspace{5pt}

\noindent{\it Proof of Theorem $\ref{ori1}$.}\hspace{2pt}
First suppose that $D_{n}^{\down}$ has eigenvalue $1$. Let $0 \neq g \in \ell^{2}(K_{n})$ be an eigenfunction of $D_{n}^{\down}$
with eigenvalue $1$. Denoting $\ol{g}$ the complex conjugate of $g$, we have $D_{n}^{\down}\ol{g}=\ol{D_{n}^{\down}g}=\ol{g}$.
Hence we can assume that $g$ is real-valued. We note that, by Theorem $\ref{ker1}$, $g$ is contained in $C^{n}(\kcal,\mb{C})$.
By $\eqref{dadj}$, Lemma $\ref{oriL}$ and the definition of $S^{\down}$, we have
\begin{equation}\label{VAL}
\frac{g(\tau_{1})}{\sqrt{\deg_{Y}(\tau_{1})}}=\eta^{\down}(\tau_{1},\tau_{2}) \frac{g(\tau_{2})}{\sqrt{\deg_{Y}(\tau_{2})}}
\end{equation}
for any $n$-down neighbors $\tau_{1},\tau_{2} \in K_{n}$.
Since our simplicial complex $\kcal$ is assumed to be strongly connected, the $n$-down graph $Y_{n}$
is connected. Hence from $\eqref{VAL}$ it follows that $g$ has no zeros. Now we set
\begin{equation}\label{BIP}
K_{n}^{o}=\{\tau \in K_{n} \,;\, g(\tau)>0\}.
\end{equation}
Let $\tau \in K_{n}^{o}$. Then $g(\ol{\tau})=-g(\tau)<0$ and hence $\ol{\tau} \not\in K_{n}^{o}$.
Since $g$ has no zeros, we have a disjoint decomposition $K_{n}=K_{n}^{o} \cup \ol{K_{n}^{o}}$.
Suppose that $\tau_{1},\tau_{2} \in K_{n}^{o}$ are $n$-down neigbors.
By $\eqref{VAL}$, $\eta^{\down}(\tau_{1},\tau_{2})>0$, and hence $\eta^{\down}(\tau_{1},\tau_{2})=1$.
Therefore, $\kcal$ is totally non-coherently orientable.

When $g \neq 0$ is a real-valued eigenfunction of $D_{n}^{\down}$ with eigenvalue $-1$, the equation $\eqref{VAL}$ becomes
\begin{equation}\label{VAL2}
\frac{g(\tau_{1})}{\sqrt{\deg_{Y}(\tau_{1})}}=-\eta^{\down}(\tau_{1},\tau_{2}) \frac{g(\tau_{2})}{\sqrt{\deg_{Y}(\tau_{2})}}
\end{equation}
for any $n$-down neighbors $\tau_{1},\tau_{2} \in K_{n}$.
Then, from this it follows that $g$ does not have zeros. We define $K_{n}^{o} \subset K_{n}$ by $\eqref{BIP}$.
The same argument as above, we have a disjoint decomposition $K_{n}=K_{n}^{o} \cup \ol{K_{n}^{o}}$.
If $\tau_{1}, \tau_{2} \in K_{n}^{o}$ are $n$-down neighbors, this time $\eqref{VAL2}$ shows
$-\eta^{\down}(\tau_{1},\tau_{2})>0$ and hence $\eta^{\down}(\tau_{1},\tau_{2})=-1$, showing that $\kcal$ is coherently orientable.

Conversely, suppose that $\kcal$ is totally non-coherently orientable. Let $K_{n}=K_{n}^{o} \cup \ol{K_{n}^{o}}$ be
a totally non-coherent orientation. We define $g \in \ell^{2}(K_{n})$ by
\[
g(\tau)=
\begin{cases}
\sqrt{\deg_{Y}(\tau)} & (\tau \in K_{n}^{o}), \\
-\sqrt{\deg_{Y}(\tau)} & (\tau \in \ol{K_{n}^{o}}).
\end{cases}
\]
Then $g \in C^{n}(\kcal,\mb{C})$ and it is easy to show that $g$ satisfies $\eqref{VAL}$.
Since $\eqref{VAL}$ is equivalent to the equation $S^{\down}d_{Y_{n}}^{*}g=d_{Y_{n}}^{*}g$,
Lemma $\ref{oriL}$ shows that $g$ is an eigenfunction of $D_{n}^{\down}$ with eigenvalue $1$.
When $\kcal$ is coherently orientable, we define the function $g$ in the same way as above by using
a coherent orientation $K_{n}^{o}$. This time $g$ satisfies $\eqref{VAL2}$ and hence $D_{n}^{\down}$
has eigenvalue $-1$.
\hfill$\square$

\vspace{10pt}

Combining Theorem $\ref{ori1}$ and Corollary $\ref{cohom1}$, we have the following.

\begin{cor}\label{cohom2}
Suppose that for each $[\tau] \in \scal_{n}$ has exactly $n+1$ down neighbors.
Then the simplicial complex $\kcal$ is coherently orientable if and only if the $n$-th cohomology group $H^{n}(\kcal,\mb{C})$
does not vanish.
\end{cor}

A topological space having a simplicial subdivision is called an $n$-dimensional homology manifold
if the homology groups of any of its link with integer coefficients are
isomorphic to the homology group of $(n-1)$-dimensional sphere.
For an arc-wise connected homology manifold $M$ with a simplicial subdivision $\kcal=(V,\scal)$,
it is well-known (\cite{Sp}) that, for any $(n-1)$ simplex $[\tau] \in \scal_{n-1}$,
there is exactly two $n$-dimensional simplices which contains $[\tau]$ as a face.
In this case each $[\sigma] \in \scal_{n}$ has exactly $n+1$ down neighbors.
Therefore Corollary $\ref{cohom2}$ can be applied.
It is well-known that the rank of $H_{n}(M,\mb{Z})$ equals that of $H^{n}(M,\mb{Z})$.
Hence, we have $H_{n}(M,\mb{C}) \cong H^{n}(M,\mb{C})$.
Therefore, by Corollary $\ref{cohom2}$, $M$ is orientable if and only if $H_{n}(\kcal,\mb{C}) \neq 0$,
which is a well-known fact on the orientability of homological manifolds.
In this case the eigenvalue $-1$ of $D_{n}^{\down}$ is simple.

We note that for infinite simplicial complexes, Theorem $\ref{ori1}$ can not hold
because we have the following.
\begin{prop}\label{NG}
Let $\kcal=(V,\scal)$ be an infinite simplicial complex. Then $D_{q}^{\down}$ can not have eigenvalue $\pm 1$.
\end{prop}
\begin{proof}
Suppose contrary that $D_{q}^{\down}$ has eigenvalue $1$. Let $g \in \ell^{2}(K_{q})$ be an eigenfunction
of $D_{q}^{\down}$ with eigenvalue $1$.
We note that Lemma $\ref{oriL}$ still holds for infinite simplicial complexes.
Thus, $\ref{eqE}$ and hence $\eqref{VAL}$ hold.
We define $h \in \ell^{2}(K_{q})$ by
\[
h(\tau)=\frac{g(\tau)}{\sqrt{2\deg_{Y}(\tau)}} \quad (\tau \in K_{q}).
\]
Then we have
\begin{equation}\label{con1}
h(\tau)= \pm h(\tau')
\end{equation}
when $\tau$ and $\tau'$ are adjacent.
We fix a vertex $\tau_{0} \in K_{q}$ such that $g(\tau_{0}) \neq 0$.
We set
\[
c_{0}=h(\tau_{0})=\frac{g(\tau_{0})}{\sqrt{2\deg_{Y}(\tau_{0})}}.
\]
Since the down graph $Y_{q}$ is connected by our assumtion in Subsection $\ref{AS}$,
$\eqref{con1}$ shows that $h(\tau)=\pm c_{0}$ for any $\tau \in K_{q}$.
But then $h$ can not be an $\ell^{2}$-function on $K_{q}$,
a contradiction. The assertion for $-1$ follows from the same discussion with $\eqref{VAL2}$.
\end{proof}

\subsection{An example, infinite cylinder}\label{CYL}

In the next section, some of examples of finite simplicial complexes will be given.
Therefore, in the rest of this section, we shall compute the spectrum of
discriminants for one example of infinite simplicial complexes, an infinite cylinder.
The triangulation we use is depicted in Figure $\ref{cylT}$.
\begin{figure}[htb]
     \scalebox{0.40}{\includegraphics{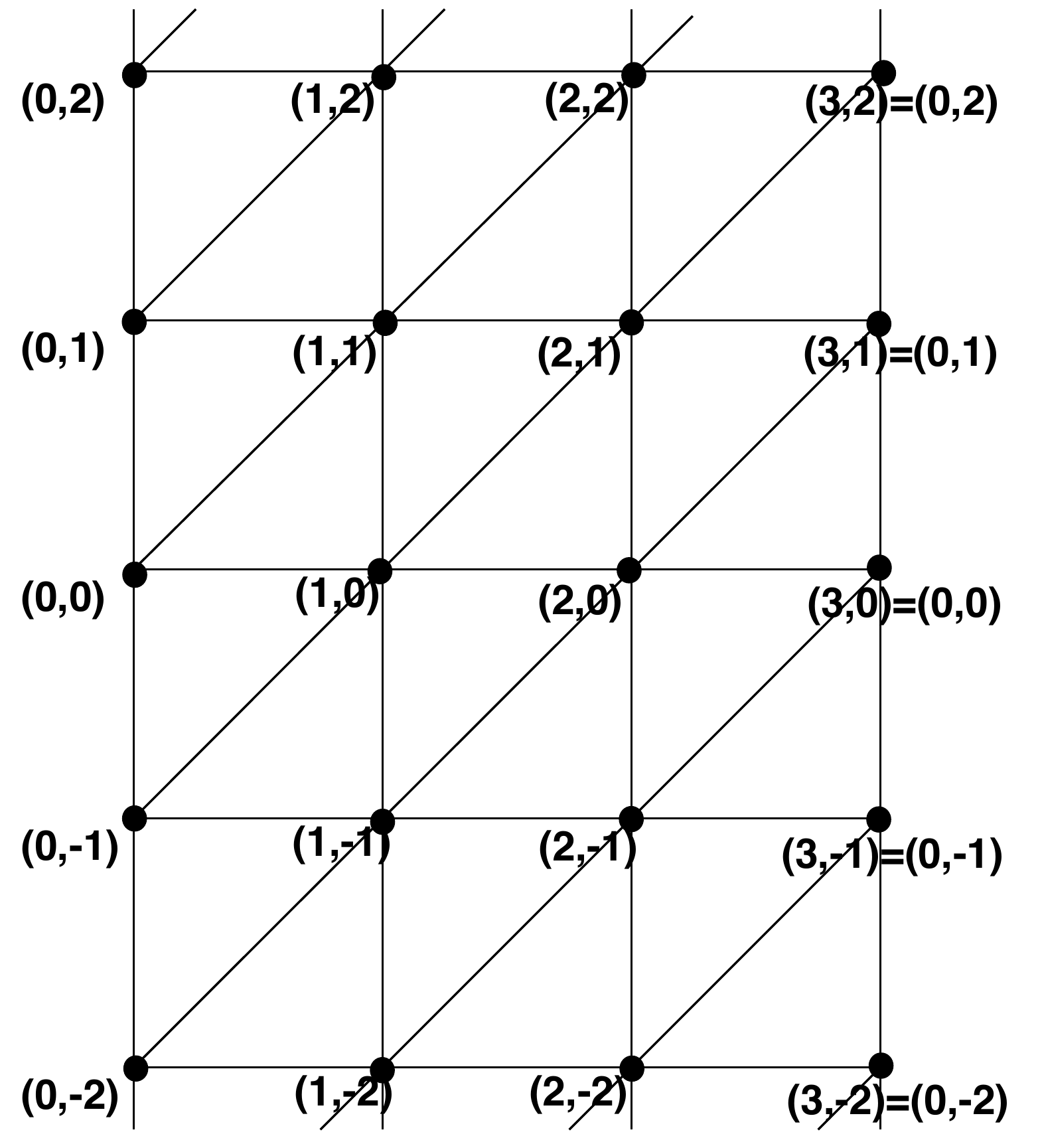}}
     \caption{Infinite cylinder}
     \label{cylT}
 \end{figure}
Since the spectrum of up-discriminants can be red of (except zero) by Corollary $\ref{anoC}$,
we only consider the down-discriminants and down-Grover walks.
\begin{prop}\label{cylP}
For spectra of the down-discriminants and down-Grover walks, the following hold.
\begin{enumerate}
\item We have $\spec(D_{2}^{\down})=[-1,1]$ all of which is continuous spectrum except for zero.
Zero is an eigenvalue of $D_{2}^{\down}$.
Hence $\spec(U_{2}^{\down})=S^{1}$ all of which is continuous spectrum except for $\pm i$ and $\pm 1$.
$\pm i$ is eigenvalues of $U_{2}^{\down}$.
\item We have $\spec(D_{1}^{\down})=\{-1/5,0\} \cup (-1/5,7/10]$, and $(-1/5,7/10]$ is continuous
spectrum and $\{-1/5,0\}$ are eigenvalues. There are no other eigenvalue of $D_{1}^{\down}$.
Hence $\spec(U_{1}^{\down})$ consists of $\{\lambda \in S^{1} \,;\, \re(\lambda) \in \{-1/5,0\} \cup (-1/5,7/10]\}$ and
possibly $\{\pm 1\}$.
A point $\lambda \in S^{1}$ is a continuous spectrum when $\re(\lambda) \in (-1/5,7/10]$ and
is an eigenvalue when $\re(\lambda)=-1/5$ or $\re(\lambda)=0$.
\end{enumerate}
\end{prop}

\begin{rem}
We remark that in Proposition $\ref{cylP}$, $\pm 1$ might be eigenvalues of $U_{1}^{\down}$, $U_{2}^{\down}$.
It can not be specified whether they are eigenvalues of down-Grover walks or not by the proof below.
See \cite{HKSS}. For $U_{1}^{\down}$, $\pm 1$ can not be continuous spectrum because it is isolated.
As is made clear in the following, the eigenvalue $0$ of the down-discriminants comes from the part $C^{q}_{+}(\kcal,\mb{C})$
for both of the case $q=1,2$.
\end{rem}

\begin{proof}
To prove Proposition $\ref{cylP}$,
we use the coordinates $(l,n)$ with $l \in \mb{Z}_{3}=\mb{Z}/3\mb{Z}$ and $n \in \mb{Z}$.
We number the $2$-simplices as
\[
\rho(l,n)=\ispa{(l,n)(l,n+1)(l+1,n+1)},\quad
\sigma(l,n)=\ispa{(l,n)(l+1,n+1)(l+1,n)}.
\]
We also number the $1$-simplices as
\[
a(l,n)=\ispa{(l,n)(l+1,n+1)},\quad
b(l,n)=\ispa{(l,n)(l,n-1)},\quad c(l,n)=\ispa{(l,n)(l-1,n)}.
\]
We note that $K_{2}^{o}=\{\rho(l,n),\sigma(l,n)\}$ is a coherent orientation of $K_{2}^{o}$.
However the orientation $K_{1}^{o}=\{a(l,n),\,b(l,n),\,c(l,n)\}$ of $K_{1}$ is not coherent.

Since  the space $C_{+}^{q}(\kcal,\mb{C})$ is contained in the kernel of $D_{q}^{\down}$,
it would be enough to consider the restriction of $D_{q}^{\down}$ to the space $C^{q}(\kcal,\mb{C})$.
Therefore, it is enough to consider the reduced down-graph $Y_{q}$ by the above orientation.
With the above orientation, $D_{2}^{\down}$ is written as
\[
\begin{split}
(D_{2}^{\down}f)(\rho(l,n)) & = -\frac{1}{3}\left[
f(\sigma(l,n))+f(\sigma(l,n+1)) +f(\sigma(l-1,n))
\right], \\
(D_{2}^{\down}f)(\sigma(l,n)) & = -\frac{1}{3}\left[
f(\rho(l,n))+f(\rho(l,n-1)) +f(\rho(l+1,n))
\right].
\end{split}
\]
A formula for $D_{1}^{\down}$ becomes very long because the reduced $1$-down graph $Y_{1}^{\down}$,
which is regular, has degree $10$. One of it is given by
\[
\begin{split}
(D_{1}^{\down}g)(a(l,n))  = & \frac{1}{10}\left[
-g(a(l-1,n-1)) -g(a(l+1,n+1))
\right.\\
& -g(b(l+1,n+1))+g(b(l+1,n+2)) -g(b(l,n+1))+g(b(l,n))  \\
& \left.
-g(c(l+1,n)) +g(c(l+2,n+1) -g(c(l+1,n+1)) +g(c(l,n))
\right].
\end{split}
\]
We note that the reduced down graphs $Y_{q}^{r}$ ($q=1,2$) is equipped with a free $\mb{Z}_{3} \times \mb{Z}$-action
which makes $Y_{q}^{r}$ a {\it crystal lattice} (\cite{KS}).
Furthermore, the discriminants are commutative with the action of $\mb{Z}_{3} \times \mb{Z}$.
First we consider the $2$-down discriminant.
Let $\wh{D}$ be the bounded operator on $L^{2}(S^{1},\ell^{2}(V_{2}))$, where
$V_{2}$ is a set
of cardinality $6$. More concretely, we write
\[
V_{2}=\{\rho_{l},\sigma_{l} \,;\, l \in \mb{Z}_{3}\},
\]
obtained by conjugating $D_{2}^{\down}$ by the Fourier transform.
The operator $\wh{D}$ is a multiplication by certain matrix-valued function $\wh{D}(z)$ in $z \in S^{1}$.
Then the spectrum of $D_{2}^{\down}$ is the union of eigenvalues of $\wh{D}(z)$ for all $z \in S^{1}$,
and the eigenvalues of $D_{2}^{\down}$ are that of $\wh{D}(z)$ which does not depend on $z \in S^{1}$.
(Similar property holds for $D_{1}^{\down}$.) Therefore, it would be enough to compute the eigenvalue of $\wh{D}(z)$.

We identify $C^{2}(\kcal,\mb{C})$ with $\ell^{2}(\mb{Z}, \ell^{2}(V_{2}))$ in a natural way.
Then we have
\[
\begin{split}
(D_{2}^{\down}f)(n)(\rho_{l}) & =-\frac{1}{3}\left[
f(n)(\sigma_{l}) +f(n+1)(\sigma_{l})+f(n)(\sigma_{l-1})
\right],\\
(D_{2}^{\down}f)(n)(\sigma_{l}) & =
-\frac{1}{3}\left[
f(n)(\rho_{l}) +f(n-1)(\rho_{l})+f(n)(\rho_{l+1})
\right],
\end{split}
\]
The matrix-valued function $\wh{D}(z)$ in $z \in S^{1}$ is then given by
\[
\wh{D}(z)=-\frac{1}{3}
\begin{pmatrix}
0 & \Omega+(1+z^{-1})I \\
\Omega^{2} +(1+z)I & 0
\end{pmatrix},
\]
where $\Omega$ is a $3 \times 3$ permutation matrix given by
\[
\Omega=
\begin{pmatrix}
0 & 0 & 1 \\
1 & 0 & 0 \\
0 & 1 & 0
\end{pmatrix}.
\]
It would be easy to compute the eigenvalues of $\wh{D}(z)$ for each $z \in S^{1}$.
The eigenvalues of $\wh{D}(z)$ ($z=e^{i\theta} \in S^{1}$) are
\[
\pm \frac{1}{3}\sqrt{5+4\cos \theta},\quad
\pm \frac{\sqrt{2}}{3}\sqrt{1+\cos (\theta +\pi/3)}.
\]
From this the first statement of Proposition $\ref{cylP}$ follows.

Next let us consider the $1$-down discriminant $D_{1}^{\down}$. Let $V_{1}=\{a_{l},b_{l},c_{l} \,;\,l \in \mb{Z}_{3}\}$.
Then $V_{2}$ is regarded as the set of vertices of the quotient graph $Y_{1}^{r}/\mb{Z}$ where $Y_{1}^{r}$ is the
reduced $1$-down graph.
This time, the matrix-valued function $\wh{D}(z)$ becomes $9 \times 9$-matrix.
But fortunately the graph $Y_{1}^{r}/\mb{Z}$ admits an action of $\mb{Z}_{3}$ and $\wh{D}(z)$ is still commutative with this action.
Then decomposing $\ell^{2}(V_{1})$ by the $\mb{Z}_{3}$-action and restricts $\wh{D}(z)$ on each isotypical subspaces (eigenspaces)
for this $\mb{Z}_{3}$-action, we get the $3 \times 3$-matrix
\[
\frac{1}{10}
\begin{bmatrix}
-(\mu z +\mu^{-1}z^{-1}) & (1-z^{-1})(1-\mu^{-1}z^{-1}) & (1-\mu^{-1})(1-\mu^{-1}z^{-1}) \\
(1-z)(1-\mu z) & -(z+z^{-1}) & (1-\mu^{-1}) (1-z) \\
(1-\mu) (1-\mu z) & (1-\mu) (1-z^{-1}) & -(\mu+\mu^{-1})
\end{bmatrix},
\]
where $\mu=1,\omega,\omega^{2}$ with $\omega=e^{2\pi i/3}$. The eigenvalues of each of these matrices is
\[
-\frac{1}{5} \ (\mbox{multiplicity $2$}),\quad \frac{2(1-\cos(\theta+\phi/2)\cos(\phi/2))-\cos \phi}{5}\ (\mbox{multiplicity $1$}),
\]
where $z=e^{i\theta}$ and $\phi=0,2\pi/3$ or $\phi=4\pi/3$. From this the second assertion follows.
\end{proof}
\begin{rem}
\begin{enumerate}
\item It does not seem so straightforward to compute the spectrum of $D_{2}^{\down}$ for infinite M\"{o}bius band.
This is because, at least for the triangulation similar to that we used for cylinder, there are no $\mb{Z}$-action
which is commutative with $D_{2}^{\down}$. $D_{2}^{\down}$ for M\"{o}bius band can be regarded as a perturbation
of that for cylinder. But the perturbation term is not commutative with $D_{2}^{\down}$ for cylinder.
\item We remark that the $1$-skeleton for the triangulation of cylinder we used is not coherently orientable.
Theorem $\ref{ori1}$ is only for finite simplicial complex. However, according to
the above computation for cylinder it seems that still there might be some relationship between
spectrum and orientation.
\end{enumerate}
\end{rem}

\section{Spectral symmetry for certain simplicial complexes}\label{SSS}

In this section, we continue to study properties of eigenvalues of discriminant operators.
We adopt the methods in \cite{FH} to prove symmetry properties of eigenvalues of discriminant operators for some finite simplicial complexes.

Let $\theta:K_q\rightarrow \{1,-1\}$ be a switching function such that $\theta(\tau)=\theta(\ol{\tau})$, for every $\tau, \ol{\tau}\in K_q$.
We define a function $\eta^{\theta}$ by
$\eta^{\theta}(\tau,\tau')=\theta(\tau)\eta(\tau,\tau')\theta(\tau')$.
Let $A^{\up}(\tau,\tau'):=\sqrt{deg_{X}(\tau)}\delta_{\tau\tau'}$ define the degree matrix for up graph, and
$A^{\down}(\tau,\tau'):=\sqrt{deg_{Y}(\tau)}\delta_{\tau\tau'}$ define the degree matrix for down graph,
where $\delta_{\tau\tau'}=1$ if $\tau=\tau'$, otherwise, $\delta_{\tau\tau'}=0$.
For a switching function $\theta,$ let $S^{\theta}$ denote the diagonal matrix defined as $S^{\theta}(\tau,\tau'):=\theta(\tau)\delta_{\tau\tau'}.$
Note that $(S^{\theta})^{-1}=S^{\theta}$ and $S^{\theta}A^{\up}=A^{\up}S^{\theta}$, $S^{\theta}A^{\down}=A^{\down}S^{\theta}$.
We then define the operator $(D_{q}^{\up})^{\theta}$ and $(D_{q}^{\down})^{\theta}$ by
$$
2(q+1)(D_q^{\up})^{\theta}=S^{\theta}(A^{\up})^{-1}\eta^{\up} (A^{\up})^{-1}S^{\theta}
$$
$$
2(D_q^{\down})^{\theta}=S^{\theta}(A^{\down})^{-1}\eta^{\down} (A^{\down})^{-1}S^{\theta}.
$$
We also note that we have
$$
2(q+1)D_q^{\up}=(A^{\up})^{-1}\eta^{\up} (A^{\up})^{-1}
$$
$$
2D_q^{\down}=(A^{\down})^{-1}\eta^{\down} (A^{\down})^{-1}.
$$
These definitions imply the following lemma.
\begin{lem}\label{symlem}
Let $\kcal$ be a finite simplicial complex and $\theta:K_q\rightarrow \{1,-1\}$ a switching function. Then the switched
$D^{\theta}$ has the same eigenvalue as $D,$ i.e.
$$
\sigma(D^{\theta})=\sigma(D).
$$
where $D$ represent $D_{q}^{\up}$ or $D_{q}^{\down}$.
\end{lem}
For two $n \times n$ matrices $A$ and $B$, we write $A \simeq B$ if $B$ is obtained by a sequence
of changes each of which replaces the $i$-th row and the $j$-th row and also
replaces the $i$-th column and the $j$-th column.
If $A\simeq B$, then the eigenvalues of $A$ and $B$ with their multiplicities are the same.
Moreover, we have the following proposition
\begin{prop}\label{sym}
Let $\kcal=(V,\scal)$ be a finite simplicial complex. If there is a switching function $\theta:K_q\rightarrow \{1,-1\}$
such that $D^{\theta}\simeq -D,$ then the eigenvalue of $D$ is symmetric, that is $\sigma(D)=-\sigma(D).$
\end{prop}
\begin{proof}
Combing Lemma \ref{symlem},
$$\sigma(D)=\sigma(D^{\theta})=-\sigma(D).$$
This proves the assertion.
\end{proof}
Similar to \cite{FH}, we have the following
\begin{thm}\label{symthm}
There exists a function $\theta:K_{q} \to \{\pm 1\}$ such that $\theta(\ol{\tau})=\theta(\tau)$ and
$D^{\theta}=-D$ if and only if the down graph $Y_{q}$ is bipartite, where for simplicity, we write $D=D_{q}^{\down}$.

Therefore, if $Y$ is bipartite, then eigenvalues of $D$ is symmetric about the origin.
\end{thm}
Theorem \ref{symthm} is a special case of Proposition \ref{sym}. We shall give its proof for completeness.
\begin{proof}
For simplicity we write
\[
D(\sigma,\sigma')=\frac{\eta(\sigma,\sigma')}{\sqrt{\deg(\sigma)\deg(\sigma')}},
\]
and if $\sigma$ and $\sigma'$ are not down neighbors, then $D(\sigma,\sigma')=0$.
For any $\sigma' \in K_{q}$, we define a function $\delta_{\sigma'} \in C^{-}$ by
\[
\delta_{\sigma'}(\sigma)=
\begin{cases}
1 & \mbox{($\sigma=\sigma'$)}, \\
-1 & \mbox{($\sigma=\ol{\sigma'}$)}, \\
0 & \mbox{(otherwise)}.
\end{cases}
\]
Then $\{\delta_{\sigma'}\}_{\sigma' \in K_{q}^{o}}$ is a basis of $C^{-}$ for any orientation $K_{q}^{o} \subset K_{q}$
and $\delta_{\ol{\sigma'}}=-\delta_{\sigma'}$.
Take $\theta:K_{q} \to \{\pm 1\}$, $\theta \in C^{+}$. Then,
\[
(D\delta_{\sigma'})(\sigma)=D(\sigma,\sigma'),\quad
(D^{\theta}\delta_{\sigma'})(\sigma)=\theta(\sigma)D(\sigma,\sigma')\theta(\sigma').
\]
Therefore, $D^{\theta}=-D$ if and only if
$\theta(\sigma)\theta(\sigma')=-1$ for any down neighbors $\sigma$, $\sigma'$.

Suppose that there exists a function $\theta \in C^{+}$ with values in $\{\pm 1\}$
such that $D^{\theta}=-D$. Then set $A=\theta^{-1}(1)$, $B=\theta^{-1}(-1)$.
Then $K_{q}$ is written as a disjoint union $K_{q}=A \cup B$.
Take $\sigma, \sigma' \in K_{q}$ such that $\sigma$ and $\sigma'$ are down neighbors.
Then by the above discussion, we have $\theta(\sigma)\theta(\sigma')=-1$. Therefore,
if $\sigma \in A$, then $\sigma' \in B$ and if $\sigma \in B$ then $\sigma' \in A$. Hence
the disjoint decomposition $K_{q}=A \cup B$ is a bipartition.

Next, assume that the graph $Y$ is bipartite. Then, by the definition of the down graph $Y$,
there exists a decomposition $K_{q}=A \cup B$ such that, if $\sigma$, $\sigma'$ are down neighbors,
then $\sigma \in A$ implies $\sigma' \in B$ and vice versa.
So now we define a function $\theta:K_{q} \to \{\pm 1\}$ by
\[
\theta(\sigma)=
\begin{cases}
1 & (\sigma \in A), \\
-1 & (\sigma \in B).
\end{cases}
\]
Take two $\sigma,\sigma' \in K_{q}$ which are down neighbors.
If $\sigma \in A$ then $\sigma' \in B$ and the definition of $\theta$ shows $\theta(\sigma)\theta(\sigma')=-1$.
The same conclusion holds also for the case $\sigma \in B$, $\sigma' \in A$.
Therefore, by the discussion of the above, we see $D^{\theta}=-D$, and hence
in this case $D$ has symmetric eigenvalues.
\end{proof}

\subsection{Examples.}
We give here some examples of eigenvalues of discriminants.
The first two examples are direct application of Theorem \ref{symthm}.
\subsubsection{Cylinder}\label{cylin}
For a cylinder with triangular decomposition drawn in Figure $\ref{f1}$,
\begin{figure}[htb]
     \scalebox{1.00}{\includegraphics{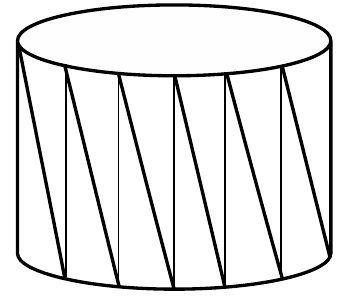}}
     \caption{cylinder}
     \label{f1}
 \end{figure}
we consider two dimensional down-Grover walk on this simplicial complex. If $|S_2|=2m$,
according to the prove of Theorem $4.3$ in \cite {HJa}, the eigenvalue set of $D_2^{\down}$ is equal to
 $$\{-cos(\frac{2j\pi}{2m})|j=0,1,...,2m-1\}\cup \{0\},$$
which is symmetric about the origin.
In fact, the corresponding down graph is bipartite.
Indeed, we can find there are two kinds $2$-simplices, we call them up and down triangles respectively, see Figure $\ref{f2}$.
 \begin{figure}[h]%[htb]
     \scalebox{1.00}{\includegraphics{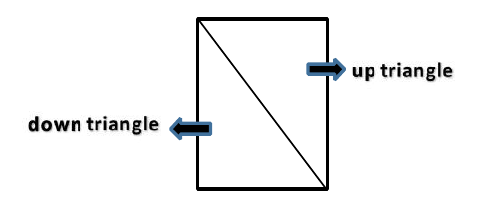}}
     \caption{up and down triangle}
     \label{f2}
 \end{figure}
Then we can divide the simplicial faces into two parts, $V_1$ consists all the up triangles
and $V_2$ consists all the down triangles.

\subsubsection{M\"{o}bius band}\label{MB}

When the Mobius band is subdivided as shown in Figure $\ref{f3}$,
\begin{figure}[h]%[htb]
     \scalebox{1.00}{\includegraphics{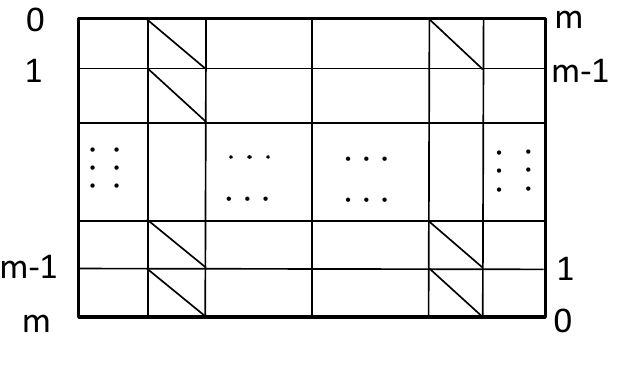}}
     \caption{M\"{o}bius band}
     \label{f3}
 \end{figure}
the corresponding down graph is bipartite, and hence eigenvalues of $D_{2}^{\down}$ are symmetric about the origin.

Here is a natural question: are there any examples that the corresponding graph are not bipartite,
but its eigenvalues of discriminant operator are symmetric. The next two examples will give an answer to this question.
\subsubsection{Sphere}\label{SP}
We consider the sphere with subdivision shown in Figure $\ref{f4}$, namely a boundary of $3$-dimensional
simplex.
\begin{figure}[h]%[htb]
     \scalebox{1.00}{\includegraphics{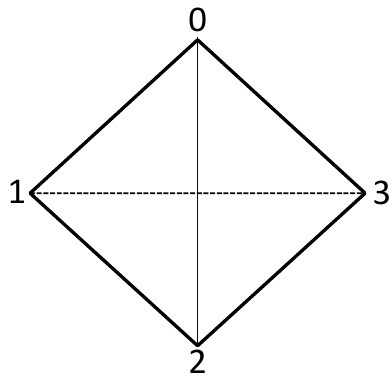}}
     \caption{Non-bipartite example for up-Grover walks}
     \label{f4}
 \end{figure}
There are four $2$-dimensional simplices, that is,
$\sigma_1=(012),\sigma_2=(123),\sigma_3=(013),\sigma_4=(023)$ and six $1$-dimensional simplices,
$\tau_1=(01),\tau_2=(02),\tau_3=(12),\tau_4=(13),\tau_5=(23),\tau_6=(03)$.
We will consider $1$-dimensional up-Grover walk.
By computation, the eigenvalue set of $D_{1}^{\up}$ is $\{\frac{1}{2},-\frac{1}{2},0\}$,
where the eigenvalue $0$ comes from the subspace $C_{+}^{1}(\kcal,\mb{C})$.
Indeed, the matrix of $D_{1}^{\up}$ restricted to these $1$-dimensional simplices with the given orientation is shown below:
\begin{table}[h]%[htb]
 $D_{1}^{\up}=$
 \begin{tabular}{c|cccccc}
    &$\tau_1$&$\tau_2$&$\tau_3$&$\tau_4$&$\tau_5$&$\tau_6$\\
   \hline
   $\tau_1$&0&-1/4&1/4&1/4&0&-1/4\\
   $\tau_2$&-1/4&0&-1/4&0&1/4&-1/4\\
   $\tau_3$&1/4&-1/4&0&-1/4&1/4&0\\
   $\tau_4$&1/4&0&-1/4&0&-1/4&-1/4\\
   $\tau_5$&0&1/4&1/4&-1/4&0&-1/4\\
   $\tau_6$&-1/4&-1/4&0&-1/4&-1/4&0\\
 \end{tabular}
 \end{table}
If we choose the switching function as $\theta(\tau_1)=-1,\theta(\tau_2)=1,\theta(\tau_3)=-1,
\theta(\tau_4)=1,\theta(\tau_5)=-1,\theta(\tau_6)=-1$,
then $(D_{1}^{\up})^{\theta}$ is shown as follows, from which one can find that $-D_{1}^{\up} \simeq (D_{1}^{\up})^{\theta}$,
$-D_{1}^{\up}$ is obtained by the change replacing the $3$-rd row and the $6$-th row and also $3$-rd column and
$6$-th column in $(D_{1}^{\up})^{\theta}$.
\begin{table}[h]%[htb]
 $(D_{1}^{\up})^{\theta}=$
 \begin{tabular}{c|cccccc}
    &$\tau_1$&$\tau_2$&$\tau_3$&$\tau_4$&$\tau_5$&$\tau_6$\\
   \hline
   $\tau_1$&0&1/4&1/4&-1/4&0&-1/4\\
   $\tau_2$&1/4&0&1/4&0&-1/4&1/4\\
   $\tau_3$&1/4&1/4&0&1/4&1/4&0\\
   $\tau_4$&-1/4&0&1/4&0&1/4&1/4\\
   $\tau_5$&0&-1/4&1/4&1/4&0&-1/4\\
   $\tau_6$&-1/4&1/4&0&1/4&-1/4&0\\
 \end{tabular}
 \end{table}

\newpage

\subsubsection{An example for down-Grover walks}\label{dGWE}
In this example, we consider $2$-dimensional down-Grover walk for a simplicial complex depicted in Figure $\ref{f5}$.
The $2$-dimensional simplices are $\tau_1=(012),\tau_2=(214),\tau_3=(134),\tau_4=(013),\tau_5=(213)$.
\begin{figure}[h]%[htb]
     \scalebox{1.00}{\includegraphics{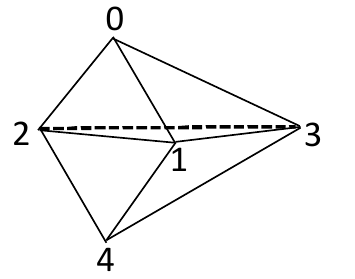}}
     \caption{Example for down Grover walks}
     \label{f5}
 \end{figure}
$D_2^{\down}$ restricted to these $2$-dimensional simplices with the given orientation is shown as follows:
\begin{table}[htb]
 $D_2^{down}=$
 \begin{tabular}{c|ccccc}
    &$\tau_1$&$\tau_2$&$\tau_3$&$\tau_4$&$\tau_5$\\
   \hline
   $\tau_1$&0&$-\frac{1}{3}$&0&$\frac{1}{3}$&$-\frac{1}{2\sqrt{3}}$\\
   $\tau_2$&$-\frac{1}{3}$&0&$-\frac{1}{3}$&0&$\frac{1}{2\sqrt{3}}$\\
   $\tau_3$&0&$-\frac{1}{3}$&0&$\frac{1}{3}$&$\frac{1}{2\sqrt{3}}$\\
   $\tau_4$&$\frac{1}{3}$&0&$\frac{1}{3}$&0&$\frac{1}{2\sqrt{3}}$\\
   $\tau_5$&$-\frac{1}{2\sqrt{3}}$&$\frac{1}{2\sqrt{3}}$&$\frac{1}{2\sqrt{3}}$&$\frac{1}{2\sqrt{3}}$&0\\
 \end{tabular}
 \end{table}
If we choose the switching function $\theta$ as $\theta(\tau_1)=\theta(\tau_2)=\theta(\tau_3)=
 \theta(\tau_4)=-1,\theta(\tau_5)=1$, then $(D_{1}^{\up})^{\theta}$ is given below.
 \begin{table}[h]%[htb]
 $(D_{2}^{\down})^{\theta}=$
 \begin{tabular}{c|ccccc}
    &$\tau_1$&$\tau_2$&$\tau_3$&$\tau_4$&$\tau_5$\\
   \hline
   $\tau_1$&0&$-\frac{1}{3}$&0&$\frac{1}{3}$&$\frac{1}{2\sqrt{3}}$\\
   $\tau_2$&$-\frac{1}{3}$&0&$-\frac{1}{3}$&0&$-\frac{1}{2\sqrt{3}}$\\
   $\tau_3$&0&$-\frac{1}{3}$&0&$\frac{1}{3}$&$-\frac{1}{2\sqrt{3}}$\\
   $\tau_4$&$\frac{1}{3}$&0&$\frac{1}{3}$&0&$-\frac{1}{2\sqrt{3}}$\\
   $\tau_5$&$\frac{1}{2\sqrt{3}}$&$-\frac{1}{2\sqrt{3}}$&$-\frac{1}{2\sqrt{3}}$&$-\frac{1}{2\sqrt{3}}$&0\\
 \end{tabular}
 \end{table}
We can find that $-D_{2}^{\down}$ is obtained by the change replacing the $2$-nd row and the $4$-th row and
also the $2$-nd column and the $4$-th column from $(D_{2}^{\down})^{\theta}$.
By a direct computation, we can see that
the eigenvalue set of $D_2^{down}$ is $\{\pm\frac{2}{3},0,\pm\frac{\sqrt{3}}{3}\}$.

\subsubsection{Eigenvalues of $(n-1)$-dimensional simplex.}\label{SPE}

Let $\kcal=(V,\scal)$ be the $(n-1)$-dimensional simplex, that is $V=\{1,2,\ldots,n\}$ and $\scal=2^{V}$.
In this case, $\deg_{X_{q}}(\tau)=n-q-1$ and $\deg_{Y_{q}}(\tau)=(q+1)(n-q-1)$.
For any $\sigma \in K_{q+1}$, we define a function $f_{\sigma} \in C^{q}(\kcal,\mb{C})$ by
\[
f_{\sigma}(\tau)=\sgn(\sigma,\tau).
\]
Note that by definition, $f_{\sigma}(\tau)=0$ if $[\tau]$ is not a face of $[\sigma]$.
It is pointed out in \cite{HJa} that there are $\binom{n-1}{q+1}$ linearly independent functions of the form $f_{\sigma}$.
Indeed, we fix a vertex $a \in V$. Then the $(q+1)$-dimensional simplices that contains $a$ as a vertex
cover all of the $q$-dimensional simplices in $\kcal$.
\begin{prop}\label{sx}
Let $\kcal=(V,\scal)$ is the $(n-1)$-dimensional simplex.
Then the eigenvalues of the restriction of
$D_{q}^{\up}$ on $C^{q}(\kcal,\mb{C})$ is $1/(n-q-1)$ and $-1/(q+1)$.
The multiplicity of $1/(n-q-1)$ is $\binom{n-1}{q+1}$ and the multiplicity of $-1/(q+1)$ is $\binom{n-1}{q}$.
\end{prop}
\begin{proof}
The proof is almost the same as that of Theorem 4.1 in \cite{HJa}.
But, we give here the details of proof for completeness.
First of all, we shall prove that $D_{q}^{\up}f_{\sigma}=\frac{1}{n-q-1}f_{\sigma}$.
We take $\sigma \in K_{q+1}$ and $\tau \in K_{q}$ and consider the following expression.
\begin{equation}\label{sgnP1}
(D_{q}^{\up}f_{\sigma})(\tau)=
\frac{1}{2(q+1)(n-q-1)} \sum_{\tau' \in K_{q} \, ;\, (\tau,\tau' ) \in E(X_{q})}
\sgn(\sigma_{q+1}(\tau,\tau'),\tau) \sgn(\sigma_{q+1}(\tau,\tau'),\tau') \sgn (\sigma,\tau'),
\end{equation}
where $[\sigma_{q+1}(\tau,\tau')]$ is a $q+1$-simplex that contain $[\tau]$ and $[\tau']$ as faces.
According to the above expression, it is enough to consider the case where
$\sigma$ contains a $q$-face $\tau'$ such that $(\tau,\tau') \in E(X_{q})$, otherwise the summation vanishes.
Suppose first that $[\tau] \subset [\sigma]$. If $\tau' \in K_{q}$, $(\tau,\tau') \in E(X_{q})$ and $[\tau'] \subset [\sigma]$,
then $[\tau]$ and $[\tau']$ is contained in both of $[\sigma]$ and $[\sigma_{q+1}(\tau,\tau')]$.
Hence $[\sigma_{q+1}(\tau,\tau')]=[\sigma]$. Thus, the equation $\eqref{sgnP1}$ becomes
\[
\begin{split}
(D_{q}^{\up}f_{\sigma})(\tau) & =\frac{1}{2(q+1)(n-q-1)} \sum_{\tau'\,;\,[\tau'] \subset [\sigma]}
\sgn(\sigma,\tau)\sgn(\sigma,\tau') \sgn(\sigma,\tau') \\
& =\frac{1}{n-q-1}\sgn(\sigma,\tau)=\frac{1}{n-q-1}f_{\sigma}(\tau).
\end{split}
\]
Suppose next that $[\tau]$ is not a $q$-face of $[\sigma]$.
We take a $q$-face $\tau' \in K_{q}$ of $\sigma$.
Then, it is easy to show that $[\sigma] \cap [\tau]$ is a $(q-1)$-face of each.
To compute the sum in $\eqref{sgnP1}$, we take a numbering of vertices of $\sigma$ and $\tau$ so that
$\sigma=\ispa{a_{0}\cdots a_{q-1}a_{q}a_{q+1}}$ and $\tau=\ispa{a_{0} \cdots a_{q-1}a}$ with $a \not\in [\sigma]$.
Since $[\tau'] \subset [\sigma]$ and $[\tau'] \cap [\tau]$ is a $(q-1)$-face of each,
$\tau'$ has the form $\tau'=\ispa{a_{0} \cdots a_{q-1}a_{q}}$ or $\tau'=\ispa{a_{0} \cdots a_{q-1} a_{q+1}}$.
These two simplices appears in the summation in $\eqref{sgnP1}$.
If $\tau'=\ispa{a_{0}\cdots a_{q-1} a_{q}}$ then $\sigma_{q+1}(\tau,\tau')=\ispa{a_{0}\cdots a_{q-1} a_{q}a}$ and hence
\[
\sgn(\sigma_{q+1}(\tau,\tau'),\tau) \sgn(\sigma_{q+1}(\tau,\tau'),\tau') \sgn (\sigma,\tau')=(-1)^{q} (-1)^{q+1} (-1)^{q+1}=(-1)^{q}.
\]
If $\tau'=\ispa{a_{0} \cdots a_{q-1} a_{q+1}}$ then $\sigma_{q+1}(\tau,\tau')=\ispa{a_{0} \cdots a_{q-1} a_{q+1}a}$ and
\[
\sgn(\sigma_{q+1}(\tau,\tau'),\tau) \sgn(\sigma_{q+1}(\tau,\tau'),\tau') \sgn (\sigma,\tau')=(-1)^{q} (-1)^{q+1} (-1)^{q}=(-1)^{q+1}.
\]
Therefore, the right-hand side of $\eqref{sgnP1}$ vanishes. This shows that $D_{q}^{\up}f_{\sigma}=\frac{1}{n-q-1}f_{\sigma}$.

To finish the proof we note that $D_{q}^{\up}=\frac{1}{(q+1)(n-q-1)}\lcal_{q}^{\up} -\frac{1}{q+1}I$.
Thus we have $\lcal_{q}^{\up}f_{\sigma}=nf_{\sigma}$ and $\ker \lcal_{q}^{\up}$ corresponds to the eigenspace of $D_{q}^{\up}$
with eigenvalue $-1/(q+1)$. By Theorem 3.1 in \cite{}, we see $\dim \ker \lcal_{q}^{\up}=\binom{n-1}{q}$
because all the reduced cohomology vanishes in this case.
This completes the proof.
\end{proof}
\begin{cor}\label{symSX}
For the $(n-1)$-dimensional simplex, the eigenvalues of $D_{q}^{\up}$ are symmetric about the origin
if and only if $n$ is even and $q=(n/2)-1$.
\end{cor}
We note that Example \ref{SP} could be viewed as a case of Corollary \ref{symSX} because
the simplicial complex in Example $\ref{SP}$ is a $2$-dimensional skeleton of $3$-dimensional simplex.

\section{Finding probabilities and stationary measures}\label{FP}

So far, we investigated rather discriminant operators than the unitary operators themselves.
But, importance of unitary operators is that we can define the finding probabilities.
First of all, let us give the definition of the finding probability for the modified S-quantum walk $G_{q}$
and up and down Grover walks $U_{q}^{\up}$, $U_{q}^{\down}$.
\begin{defin}
\begin{enumerate}
\item For the modified S-quantum walk $G_{q}$, the finding probability $Q_{n}(q,f;F)$ in dimension $q$
at time $n$ and at a simplex $F \in \scal_{q}$ with an initial state $f \in \ell^{2}(\wh{K}_{q})$ $(\|f\|_{\wh{K}_{q}}=1)$
is defined as
\[
Q_{n}(q,f;F)=\sum_{s \in \wh{K}_{q},\,[s]=F} \left|(G_{q}^{n}f)(s)\right|^{2}.
\]
\item For the up-Grover walk $U_{q}^{\up}$, the finding probability $P_{n}^{\up}(q,\phi;F)$ in dimension $q$
at time $n$ and at a simplex $F \in \scal_{q}$ with an initial state $\phi \in \ell^{2}(E(X_{q}))$ $(\|\phi\|_{E(X_{q})}=1)$ is
defined as
\[
P_{n}^{\up}(q,f;F)=\sum_{(\tau,\tau') \in E(X_{q}),\,[\tau]=F}
\left|[(U_{q}^{\up})^{n}\phi](\tau,\tau')
\right|^{2}
\]
\item For the down-Grover walk $U_{q}^{\down}$, the finding probability $P_{n}^{\down}(q,\phi;F)$ in dimension $q$
at time $n$ and at a simplex $F \in \scal_{q}$ with an initial state $\phi \in \ell^{2}(E(Y_{q}))$ $(\|\phi\|_{E(Y_{q})}=1)$ is defined
in the same as in $(2)$ with replacing $X_{q}$ by $Y_{q}$.
\end{enumerate}
\end{defin}

The above definitions are identical to the finding probabilities defined in \cite{HKSS} and \cite{MOS}.
It would be necessary to compare the finding probabilities defined by S-quantum walk and
the up-Grover walk. According to Theorem $\ref{sgroT}$ there are strong relationship between
$G_{q}$ and $U_{q-1}^{\up}$. Therefore, it would be rather natural to compare $Q_{n}(q,f;F)$
and $P_{n}^{\up}(q-1,g;H)$ for some initial states $f$ and $g$.
Since $D(G_{q})$ and $D_{q-1}^{\up}$ acts rather trivially on $C_{+}^{q-1}(\kcal,\mb{C})$,
it would be natural to take initial states in $C_{+}^{q-1}(\kcal,\mb{C})$.
However, it would not so straightforward to find complete relationship
between $Q_{n}(q,f;F)$ and $P_{n}^{\up}(q-1,g;H)$ because the correspondence of
eigenfunctions is not quite simple. Hence we just give one simple situation where
one can find good relationship between them.
\begin{prop}\label{FPP1}
Let $f \in C^{q-1}(\kcal,\mb{C})$ be an eigenfunction of $D_{q-1}^{\up}$ with eigenvalue $1$ satisfying \\
$\|f\|_{\ell^{2}(K_{q-1})}=1$.
We then have
\[
Q_{n}(q, {\wh \alpha_{q-1}r};F)
=\sum_{[\tau]\in \partial F}\frac{1}{\deg_{X}(\tau)}P_{n}^{\up}(q-1,{d^{*}_{X_{q-1}}f};[\tau]),
\]
where $r=\frac{\sqrt{2}}{\sqrt{q!}}f\in C^{q-1}(\kcal,\mb{C}) \subset \ell^{2}(\wh{K}_{q-1})$,
$\|r\|_{\ell^{2}(\wh {K}_{q-1})}=1,$ and $\partial F$ is the boundary of $F$.
\end{prop}

\begin{proof}
Let $f \in C^{q-1}(\kcal,\mb{C})$ be such a function as in the statement. Since $D_{q-1}^{\up}f=f$,
by Theorem $\ref{sgroT}$, (2), we have
\[
D(G_{q})r=r,
\]
where $r=\frac{\sqrt{2}}{\sqrt{q!}}f\in C^{q-1}(\kcal,\mb{C}) \subset \ell^{2}(\wh{K}_{q-1})$.
We note that $r$ satisfies $\|r\|_{\ell^{2}(\wh{K}_{q-1})}=1$.
It is straightforward to see that
\[
U_{q-1}^{\up}d^{*}_{X_{q-1}}f=d^{*}_{X_{q-1}}f,\quad
G_{q}\wh \alpha_{q-1}r=\wh \alpha_{q-1}r.
\]
The functions $d^{*}_{X_{q-1}}f$ and $\wh{\alpha}_{q-1}r$ satisfy $\|d^{*}_{X_{q-1}}f\|=\|\wh \alpha_{q-1}r\|=1$.
Then, for any $\tau \in K_{q-1}$, we have
\[
\begin{split}
P_{n}^{\up}(q-1,{d^{*}_{X_{q-1}}f};[\tau]) & =
\sum_{(\tau,\tau') \in E(X_{q})}|(d^{*}_{X_{q-1}}f)(\tau,\tau')|^{2}+\sum_{(\ol{\tau},\tau') \in E(X_{q})}|(d^{*}_{X_{q-1}}f)(\ol{\tau},\tau')|^{2} \\
&= |f(\tau)|^{2}+|f(\ol{\tau})|^{2}\\
&= 2|f(\tau)|^{2}.
\end{split}
\]
The finding probability $Q_{n}(q,\wh{\alpha}_{q-1}r;F)$ at a simplex $F \in \scal_{q}$ with the initial state $\wh{\alpha}_{q-1}r$ is
then computed as follows.
\[
\begin{split}
Q_{n}(q, {\wh \alpha_{q-1}r};F)
&= \sum_{\pi \in \mf{S}_{q+1}}|\wh \alpha_{q-1}r(s_{F}^{\pi})|^{2}\\
&= \sum_{\pi \in \mf{S}_{q+1}}\frac{1}{\deg(\nu_{q}(s_{F}^{\pi}))}|r(\nu_{q}(s_{F}^{\pi}))|^{2}\\
&=\sum_{[\tau] \in \partial F}\frac{q!}{2}\frac{1}{\deg_{X}(\tau)}(|r(\tau)|^{2}+|r(\ol{\tau})|^{2})\\
&=\sum_{[\tau] \in \partial F}\frac{1}{\deg_{X}(\tau)}(|f(\tau)|^{2}+|f(\ol{\tau})|^{2})\\
&=\sum_{[\tau] \in \partial F}\frac{1}{\deg_{X}(\tau)}P_{n}^{\up}(q-1,{d^{*}_{X_{q-1}}f};[\tau]),
\end{split}
\]
where if $F=\{a_{0},\ldots,a_{q}\} \in \scal_{q}$ then we put $s_{F}=(a_{0} \cdots a_{q}) \in \wh{K}_{q}$ in the above computation.
\end{proof}

Next, we compute the finding probability with initial state $f \in C_{+}^{q}(\kcal,\mb{C})$.
\begin{prop}\label{FPP2}
For any $f \in C_{+}^{q}(\kcal,\mb{C})\subset \ell^{2}(K_q)$ with $\|f\|_{\ell^{2}(K_q)}=1$, we have
\begin{equation}\label{FPPQ}
P_{n}^{\up}(q,\psi;[\tau])=\frac{1}{2(q+1)}\sum^{m}_{j=1}Q_{n}(q+1, h; F_{j})+\frac{q}{q+1}|f(\tau)|^{2} \quad (\tau \in K_{q}),
\end{equation}
where $\psi=\frac{1}{\sqrt{2}}(I-iS^{\up})d^{*}_{X_q}f$, $h=\frac{\sqrt{2}}{\sqrt{(q+1)!}}\wh \alpha_q f$ and
$\{F\}^{m}_{j=1}$ is the set of $(q+1)$-dimensional simplices that contain $[\tau]$ as a face.
\end{prop}

\begin{cor}\label{FPC2}
Suppose that our simplicial complex $\kcal$ is finite.
For $f \in C_{+}^{q}(\kcal,\mb{C})$ given by $f(t)\equiv f(\ol{t})\equiv\frac{1}{\sqrt{2|\scal_{q}|}}$ for any $t \in K_q$,
where $|\scal_{q}|$ denotes the cardinality of $\scal_{q}$, we have the following.
\[
\begin{gathered}
P_{n}^{\up}(q,\psi;[\tau])
=\frac{1}{2|\scal_{q}|}
+\sum_{(\tau,\tau') \in E(X_{q})}\frac{1}{2(q+1)\deg_{X}(\tau')}\frac{1}{2|\scal_{q}|}, \\
Q_{n}(q+1, h; F)
=\sum_{[t] \in \partial F}\frac{1}{\deg([t])} \frac{1}{|\scal_{q}|},
\end{gathered}
\]
where $\psi \in \ell^{2}(E(X_{q}))$ and $h \in \ell^{2}(\wh{K}_{q+1})$ is as in Proposition $\ref{FPP2}$.
For $q$-regular simplicial complex, i.e., there is some
$k>0$, such that $deg_{X}(\tau)=k$ for all $\tau \in K_q$, we see
$$
P_{n}^{\up}(q,\psi;[\tau])=\frac{1}{|\scal_{q}|}
$$
$$
Q_{n}(q+1, h;F)=\frac{q+2}{|\scal_{q}|k}
$$
\end{cor}

\vspace{10pt}

\noindent{\it Proof of Proposition $\ref{FPP2}$.}\hspace{3pt}
By Theorem \ref{ker1}, for any $f \in C_{+}^{q}(\kcal,\mb{C})$ with $\|f\|_{\ell^{2}(K_q)}=1$, we have
$D_{q}^{\up}f=0$.
Then it is straightforward to see that
$$
U_q^{\up}\psi=i\psi,
$$
where $\psi=\frac{1}{\sqrt{2}}(I-i S^{\up})d^{*}_{X_q}f$.
Now the finding probability at time $n$ at $[\tau] \in \scal_q$ with initial state $\psi$ is equal to
$$
P_{n}^{\up}(q,\psi;[\tau])=\sum_{(\tau,\tau') \in E(X_{q})}|\psi(\tau,\tau')|^{2}+\sum_{(\ol{\tau},\tau') \in E(X_{q})}|\psi(\ol{\tau},\tau')|^{2}.
$$
Next, we compute $\psi$ as follows.
\[
\begin{split}
\psi(\tau_1,\tau_2)
&= \frac{1}{\sqrt{2}}(I-i S^{\up})d^{*}_{X_q}f(\tau_1,\tau_2)\\
&= \frac{1}{2}\frac{1}{\sqrt{(q+1) \deg_{X}(\tau_1)}}f(\tau_1)\\
& \hspace{10pt}
-\frac{i}{2} \sgn(\sigma_{q+1}(\tau_1,\tau_2),\tau_1)
\sgn(\sigma_{q+1}(\tau_1,\tau_2),\tau_2)\frac{1}{\sqrt{(q+1) \deg_{X}(\tau_2)}}f(\tau_2).
\end{split}
\]
Hence we have
\[
\begin{split}
|\psi(\tau_1,\tau_2)|^{2}
&= \frac{1}{4(q+1) \deg_{X}(\tau_1)}|f(\tau_1)|^{2}+\frac{1}{4(q+1) \deg_{X}(\tau_2)}|f(\tau_2)|^{2}\\
&\hspace{10pt}
-\frac{1}{2}\times\frac{\sgn(\sigma_{q+1}(\tau_1,\tau_2),\tau_1)}{\sqrt{(q+1) \deg_{X}(\tau_1)}}
\frac{\sgn(\sigma_{q+1}(\tau_1,\tau_2),\tau_2)}{\sqrt{(q+1)\deg_{X}(\tau_2)}} \re(f(\tau_1)\ol{if(\tau_2)}).
\end{split}
\]
Since the last part of the above vanishes when we take the sum, we obtain
\begin{equation}\label{stat1}
\begin{split}
P_{n}^{\up}(q,\psi;[\tau])
&= \sum_{(\tau,\tau') \in E(X_{q})}|\psi(\tau,\tau')|^{2}+\sum_{(\ol{\tau},\tau') \in E(X_{q})}|\psi(\ol{\tau},\tau')|^{2}\\
&= \frac{1}{2}|f(\tau)|^{2}+\frac{1}{2}|f(\ol{\tau})|^{2}\\
&\hspace{10pt}
+\sum_{(\tau,\tau')}\frac{1}{2(q+1)\deg_{X}(\tau')}|f(\tau')|^{2}
\end{split}
\end{equation}
Next we compute the finding probability $Q_{n}(q+1,h;F)$ with initial state $h$.
The function $g$ given by $g=\frac{\sqrt{2}}{\sqrt{(q+1)!}}f \in C_{+}^{q}(\kcal,\mb{C}) \subset \ell^{2}(\wh {K}_{q})$ satisfies
$\|g\|_{ \ell^{2}(\wh {K}_{q})}=1$ and, by Theorem $\ref{sgroT}$, (1),
\[
D(G_q)g=-g.
\]
From this, it is easy to see that we have
\[
G_{q+1}h=-h
\]
where $h=\wh \alpha_q g$.
The finding probability at time $n$ at $F \in \scal_{q+1}$ with initial state $h$ is equal to
\[
\begin{split}
Q_{n}(q+1, h; F)
&= \sum_{\pi \in \mf{S}_{q+2}}|h(s_{F}^{\pi})|^{2}\\
&= \sum_{\pi \in \mf{S}_{q+2}}\frac{1}{\deg(\nu_{q+1}F^{\pi})}|g(\nu_{q+1}s_{F}^{\pi})|^{2}\\
&= \sum_{[t]\in \partial F}\frac{1}{\deg([t])}|g([t])|^{2}(q+1)!\\
&= \sum_{[t]\in \partial F}\frac{1}{\deg([t])}\times2|f([t])|^{2}
\end{split}
\]
Comparing with the computation on $P_{n}^{\up}(q,\psi;[\tau])$, we obtain $\eqref{FPPQ}$.
\hfill$\square$

\vspace{10pt}

In order to describe down adjacency, it is also necessary to compute finding probability for Grover walks on down graphs.
\begin{prop}\label{FPP3}
For any $f\in C_{+}^{q}(\kcal,\mb{C})$ with $\|f\|_{\ell^{2}(\wh{K}_{q})}=1$,
the finding probability for the down-Grover walk $U_{q}^{\down}$ is given by
\[
P_{n}^{\down}(q,\eta;[\tau])=\frac{1}{2}|f(\tau)|^{2}+\frac{1}{2}|f(\ol{\tau})|^{2}
+\sum_{(\tau,\tau') \in E(Y_{q})}\frac{1}{2\deg_{Y}(\tau')}|f(\tau')|^{2}.
\]
Moreover, when $f(t)\equiv f(\ol{t})\equiv\frac{1}{\sqrt{2|\scal_{q}|}}$ for any $t \in K_q$, we have
\begin{equation}\label{FPdown}
P_{n}^{\down}(q,\eta;[\tau])
=\frac{1}{2|\scal_{q}|}+\sum_{(\tau,\tau')\in E(Y_{q})}\frac{1}{2\deg_{Y}(\tau')}\frac{1}{2|\scal_{q}|},
\end{equation}
where $\eta=\frac{1}{\sqrt{2}}(I-iS^{\down})d^{*}_{Y_q}f$.
\end{prop}
\begin{proof}
For any $f\in C_{+}^{q}(\kcal,\mb{C})$, we have
$D_{q}^{\down}f=0$. Thus, we see
$$
U_q^{\down}\eta=i\eta
$$
where $\eta=\frac{1}{\sqrt{2}}(I-iS^{\down})d^{*}_{Y_q}f$.
A direct computation shows
\[
\begin{split}
\eta(\tau_1,\tau_2)
& =\frac{1}{2}\frac{1}{\sqrt{\deg_{Y}(\tau_1)}}f(\tau_1)\\
&\hspace{10pt}-\frac{i}{2}\sgn(\sigma_{q-1}(\tau_1,\tau_2),\tau_1)
\sgn(\sigma_{q-1}(\tau_1,\tau_2),\tau_2)\frac{1}{\sqrt{\deg_{Y}(\tau_2)}}f(\tau_2).
\end{split}
\]
From this, we see
\[
\begin{split}
|\eta(\tau_1,\tau_2)|^{2}
&= \frac{1}{4\deg_{Y}(\tau_1)}|f(\tau_1)|^{2}+\frac{1}{4\deg_{Y}(\tau_2)}|f(\tau_2)|^{2}\\
&\hspace{10pt} -\frac{1}{2}\times\frac{\sgn(\sigma_{q-1}(\tau_1,\tau_2),\tau_1)}{\sqrt{\deg_{Y}(\tau_1)}}
\frac{\sgn(\sigma_{q-1}(\tau_1,\tau_2),\tau_2)}{\sqrt{\deg_{Y}(\tau_2)}}\re (f(\tau_1)\ol{if(\tau_2)}).
\end{split}
\]
As in the computation in the proof of Proposition $\ref{FPP2}$, we obtain $\eqref{FPdown}$.
\end{proof}
\begin{rem}
The finding probability given in Propositions $\ref{FPP1}$, $\ref{FPP2}$ and $\ref{FPP3}$
gives examples of stationary measures for up and down Grover walks. See \cite{KK, KT} for
discussion about stationary measure.
In order to find non-trivial stationary measure,
one usually tries to find the eigenfunctions of corresponding unitary operator.
In this paper, Theorem \ref{ker1} and \ref{sgroT} give a way to find non-trivial stationary measure for any
simplicial complex for modified S-quantum walks and  up and down Grover walks. For example,
$\eqref{stat1}$ gives an explicit formula for a stationary measure given by eigenfunctions of eigenvalue $i$.
The eigenfunctions we used here are rather trivially obtained one.
Hence it would be interesting to find other eigenfunctions in some examples.
\end{rem}

\section*{Acknowledgement}
The first author was supported by China Scholarship Council to study at Tohoku University for one year.

\end{document}